\documentclass{article}

\pdfoutput=1 

\usepackage[a4paper,
 total={144mm,237mm},
 left=33mm,
 top=30mm,
 ]{geometry}
\usepackage[utf8]{inputenc}
\usepackage{mathpazo}
\usepackage{hyperref}

\usepackage[T1]{fontenc}
\usepackage{amssymb, amsmath}
\usepackage[english]{babel}
\usepackage{amsthm}
\usepackage{enumitem}
\usepackage{tikz}
\usepackage{bm}
\usepackage{csquotes}
\usepackage[shortcuts]{extdash}

\usepackage[format=plain,
            labelfont=it,
            textfont=it]{caption}
\usepgflibrary{shapes.geometric}
\usepackage{float}
\allowdisplaybreaks

 \hypersetup{breaklinks=true, 
           colorlinks=true,   
        }

\usepackage[maxnames=20]{biblatex} 
\bibliography{MSbibliography.bib}

\usepackage{subfiles}

\AtBeginBibliography{\small}
 
\newtheorem{theorem}{Theorem}[section]
\newtheorem{innercustomthm}{Theorem}
\newenvironment{customthm}[1]
  {\renewcommand\theinnercustomthm{#1}\innercustomthm}
  {\endinnercustomthm}
\newtheorem{corollary}[theorem]{Corollary}
\newtheorem{lemma}[theorem]{Lemma}
\newtheorem{prop}[theorem]{Proposition}
\newtheorem{fact}[theorem]{Fact}
\theoremstyle{remark}
\newtheorem{remark}[theorem]{Remark}
\newtheorem{notation}[theorem]{Notation}

\theoremstyle{definition}

\newtheorem{definition}[theorem]{Definition}
\newtheorem{question}[theorem]{Question}
\newtheorem{construction}[theorem]{Construction}

\author{Paolo Marimon}
\title{On the non-measurability of $\omega$-categorical Hrushovski constructions}

\graphicspath{{Pictures/}}

\newcommand{\dimeas}[1]{(\mathrm{d}({#1}), \mu(#1))}

\newcommand{\mahT}{
\raisebox{-6pt}{\includegraphics[]{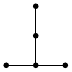}}}

\newcommand{\ptwo}{\raisebox{-1pt}{\includegraphics[]{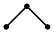}}}

\newcommand{\dtwo}{\raisebox{1pt}{\includegraphics[]{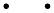}}}

\newcommand{\pthree}{\raisebox{-3pt}{
\includegraphics[]{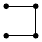}}\:}

\newcommand{\elle}{\raisebox{-3pt}{\includegraphics[]{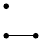}}}

\newcommand{\ptwodot}{\raisebox{-2pt}{\includegraphics[]{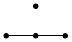}}}

\newcommand{\indep}[2]{%
  \mathrel{
    \mathop{
      \vcenter{
        \hbox{\oalign{\noalign{\kern-.3ex}\hfil$\vert$\rlap{$^\mathrm{#2}$}\hfil\cr
              \noalign{\kern-.7ex}
              $\smile$\cr\noalign{\kern-.3ex}}}
      }
    }\displaylimits_{#1}
  }
}

\begin{document}


\maketitle

\begin{abstract}
    We study $\omega$-categorical $MS$-measurable structures. Our main result is that a class of $\omega$-categorical Hrushovski constructions, supersimple of finite $SU$-rank is not $MS$-measurable. These results complement the work of Evans on a conjecture of Macpherson and Elwes. In constrast to Evans' work, our structures may satisfy independent $n$-amalgamation for all $n$. We also prove some general results in the context of $\omega$-categorical $MS$-measurable structures. Firstly, in these structures, the dimension in the $MS$-dimension-measure can be chosen to be $SU$-rank. Secondly, non-forking independence implies a form of probabilistic independence in the measure. The latter follows from more general unpublished results of Hrushovski, but we provide a self-contained proof.
\end{abstract}

\section{Introduction}

This work focuses on $\omega$-categorical $MS$-measurable structures and on the role of supersimple $\omega$-categorical Hrushovski constructions in settling some conjectures about them.  In particular, we prove that a certain class of $\omega$-categorical supersimple Hrushovski constructions of finite $SU$-rank is not $MS$-measurable. The question of whether, in general, such structures are not $MS$-measurable is still open.\\

Let $\mathcal{L}$ be a first order language. A complete $\mathcal{L}$-theory is \textbf{$\pmb{\omega}$-categorical} if it has a unique countable model up to isomorphism. We say that an $\mathcal{L}$-structure is $\omega$-categorical if its theory is. An $MS$-measurable structure has an associated dimension-measure function on its definable sets, where for the sets of a given dimension there is an invariant, additive measure satisfying Fubini's theorem with respect to the dimension and giving such sets positive measure. We give a formal definition in Definition \ref{MScond}. Standard examples of $MS$-measurable structures are pseudofinite fields, the random graph, and $\omega$-categorical $\omega$-stable structures \cite{EM}. In this article, we focus on whether $\omega$-categorical Hrushovski constructions are $MS$-measurable. These are a class of relational structures, which generalise Fra\"{i}ss\'{e} limits, and whose construction depends on a choice of a parameter $\alpha\in\mathbb{R}^{>0}$ and of a non-decreasing function $f:\mathbb{R}^{>0}\to\mathbb{R}^{>0}$. We discuss their construction in section \ref{Hrushconstr}.\\

In their review article on $MS$-measurable structures \cite{EM}, Elwes and Macpherson ask the following questions which require the study of $\omega$-categorical Hrushovski constructions:
\begin{question}\label{q1}
Is there any $\omega$-categorical supersimple theory of finite $SU$-rank which is not $MS$-measurable?
\end{question}
\begin{question}\label{q2}
Is every $\omega$-categorical $MS$-measurable structure one-based?
\end{question}
Being $MS$-measurable implies being supersimple of finite $SU$-rank \cite{EM}. Furthermore, in such structures $D$-rank, $S1$-rank and $SU$-rank are the same. Hence, the first question is simply asking whether supersimplicity and finite rank imply $MS$-measurability in an $\omega$-categorical context. \\

Supersimple $\omega$-categorical Hrushovski constructions of finite $SU$-rank are essential structures for the study of these questions. These are the only examples of $\omega$-categorical supersimple not one-based structures which we know of. Hence, finding out that any of these structures are not $MS$-measurable would answer the first question positively. Finding out that any of these structures is $MS$-measurable would answer negatively the second question. 
We already know that some $\omega$-categorical Hrushovski constructions are not $MS$-measurable from the work of Evans \cite{Measam}. There, Evans first proves that any $MS$-measurable structure must satisfy a weak form of independent $n$-amalgamation using a version by Towsner of the Hypergraph Removal Lemma \cite{Town2,Gowers,Rodl}. Then, he proves that for a class of Hrushovski constructions (with a ternary relation and $\alpha=1$), any dimension is a scaled version of the natural Hrushovski dimension. With these tools, Evans shows that some such structures do not satisfy the weak $n$-amalgamation property.\\

While Evans' dimension theorem can be generalised to other classes of Hrushovski constructions, it is possible for Hrushovski constructions to satisfy the weak amalgamation condition (and even independent $n$ amalgamation for all $n$) as long as $f$ is slow-growing enough \cite{Udiamalg}. Indeed, a natural question which arises from Evans' paper is whether $MS$-measurability for an $\omega$-categorical finite rank structure is implied by satisfying some strong enough form of independent $n$-amalgamation.\\

In this paper, we develop new tools to study $\omega$-categorical $MS$-measurable structures. These will allow us to study how an $MS$-dimension-measure would behave in an $\omega$-categorical Hrushovski construction if it was $MS$-measurable. In particular, in Section \ref{finally}, we introduce the $\omega$-categorical supersimple finite rank Hrushovski constructions $\mathcal{M}_f$ for which we prove the main theorem of this paper:

\begin{customthm}{\ref{eqnstheorem}} The structures $\mathcal{M}_f$ satisfying the conditions of Construction \ref{actualf} are not $MS$\-/measurable. Hence, there are $\omega$-categorical supersimple finite $SU$-rank structures which are not $MS$-measurable. Indeed, there are supersimple $\omega$-categorical structures of finite $SU$-rank and with independent $n$-amalgamation over finite algebraically closed sets for all $n$ which are not $MS$-measurable.
\end{customthm}
The proof of this result is contained in subsection \ref{proof!}. \\

The structure of this paper is as follows. In Section \ref{MSmeasures}, we introduce some basic notions concerning $MS$-measurable structures. We also prove the folklore theorem of Ben Yaacov that the natural notion of independence in $MS$-measurable structures corresponds to non-forking independence. In Corollary \ref{amacor}, we find a set of equations that the measure of an $MS$-measurable structure must satisfy. We also prove that, in an $\omega$-categorical context, we may take the dimension of an $MS$-measurable structure to be $SU$-rank (Corollary \ref{SUok}). This allows us to circumvent the need to prove a dimension theorem for the class of Hrushovski constructions we will consider.

Section \ref{indepmeas} yields another set of equations for the measure in $\omega$-categorical $MS$-measurable structures which show how non-forking independence yields probabilistic independence in the measure (Theorem \ref{indeptheorem}). These results are a special case of the probabilistic independence theorem in the unpublished \cite{AER}.

In Section \ref{Hrushconstr}, we introduce Hrushovski constructions and in subsection \ref{Hrusheqns} we specify the form of the equations from section \ref{MSmeasures} in the context of such structures.

In Section \ref{finally}, we introduce in Remark \ref{actualf} various supersimple $\omega$-categorical Hrushovski constructions of finite $SU$-rank, some of which satisfy independent $n$-amalgamation for all $n$. We will prove these are not $MS$-measurable in subsection \ref{proof!}. Various properties of these structures (especially supersimplicity) are proven in the Appendix.
The most relevant results for understanding the proof of our main theorem are Corollaries \ref{amacor} and \ref{SUok}, Theorems \ref{indeptheorem} and \ref{Hrush}, and subsection \ref{proof!}.\\

We assume some knowledge of model theory,  especially regarding $\omega$-categorical structures. Most of the relevant material is covered in Chapters 1 to 4 of Tent and Ziegler's book \cite{TZ}. Section \ref{indepmeas} uses some local stability theory for which the first chapter of \cite{Pillay} is sufficient. We work in countable languages and only with complete theories. We shall make frequent use of the equivalent conditions to $\omega$-categoricity from the Ryll-Nardzewski theorem. We also assume some basic knowledge of simple theories and $SU$-rank. Chapter 2 of \cite{Kimsimp} covers the relevant material.\\

Regarding notation, for a language $\mathcal{L}$, we work on an $\mathcal{L}$-structure $\mathcal{M}$ and write $M$ to denote the underlying set. The monster model for the $\mathcal{L}$-theory of $\mathcal{M}$ is denoted by $\mathbb{M}$. We use overlined lowercase letters at the beginning of the alphabet $\overline{a},\overline{b}, \dots$ to denote finite tuples from a model, and use  $\overline{x},\overline{y}, \cdots$ to denote tuples of variables. We use the non-overlined versions when speaking of $1$-tuples is sufficient. We use letters $A, B, C, \dots$ to denote (usually finite) subsets of our model. We use the greek letters $\phi, \psi, \chi\dots$ to denote $\mathcal{L}$-formulas, which we often write in the form $\phi(\overline{x}, \overline{y})$ to specify their free variables.\\

\textbf{Acknowledgements:} I would like to thank David Evans and Charlotte Kestner for their supervision and support on this project. I would also like to thank Ehud Hrushovski for sharing his notes \cite{AER} on a stronger version of the probabilistic independence theorem than the one proved in this article. The theorem greatly simplified the equations needed to prove Theorem \ref{eqnstheorem} and provided an excellent lens  to study the measures arising in $MS$-measurable structures. This paper is part of my PhD project at Imperial College London, which is supported by an Admin-Roth Scholarship.

\section{Measurable \texorpdfstring{$\omega$}{omega}-categorical structures}\label{MSmeasures}
In this section, we introduce $MS$-measurable structures and some basic facts about them. We also prove some original results on $\omega$-categorical $MS$-measurable structures.
We begin with Subsection \ref{basicdef}, where we introduce the notion of $MS$-measurability following \cite{MS} and \cite{EM}. Then, in Subsection \ref{dimind}, we prove the folklore result that dimension independence in $MS$-measurable structures corresponds to non-forking independence. Finally, in Subsection \ref{omegameas}, we find a set of equations which hold in $\omega$-categorical $MS$-measurable structures. We also show that if an $\omega$-categorical structure is $MS$-measurable, then it is $MS$-measurable with dimension given by $SU$-rank. This will allows us to avoid proving Evans' dimension theorem for the class of $\omega$-categorical Hrushovski constructions we will consider.

\subsection{Basic definitions}\label{basicdef}

Let $\mathcal{M}$ be a structure. By  $\mathrm{Def}(M)$ we mean the set of non-empty definable subsets of $M$ defined by formulas with parameters from $M$. Meanwhile, $\mathrm{Def}_{\overline{x}}(M)$ denotes the definable subsets of $M$ in the variable $\overline{x}$. For $\overline{a}$, $\overline{a}'\in M^{\vert\overline{y}\vert}$ and $A\subset M$, we write $\overline{a}\equiv_A \overline{a}'$ to say that $\overline{a}$ and $\overline{a}'$ have the same type over A.
If $A=\emptyset$, we simply write $\overline{a}\equiv \overline{a}'$.

\begin{definition} Let $X$ be any set and consider a function $g:\mathrm{Def(M)}\to X$.
For $A\subseteq M$, we say that $g$ is $A$-\textbf{invariant} if  $g(\phi(M^{\vert\overline{x}\vert}, \overline{a}))=g(\phi(M^{\vert\overline{x}\vert}, \overline{a}'))$ whenever $\overline{a}\equiv_A \overline{a}'$. When $A=\emptyset$, we say $g$ is \textbf{invariant}. The function $g$ is \textbf{definable} if for any formula $\phi(\overline{x}, \overline{y})$ and $k\in X$, the set of $\overline{a}\in M^{\vert\overline{y}\vert}$ such that $g(\phi(M^{\vert\overline{x}\vert}, \overline{a}))=k$ is definable over the empty set. We say that $g$ is \textbf{finite} if the set of values of $g(\phi(M^{\vert\overline{x}\vert}, \overline{a}))$ for $\overline{a}\in M^{\vert\overline{y}\vert}$ is finite.\\
To avoid cumbersome notation, when the model $\mathcal{M}$ is clear, we sometimes write $g(\phi(\overline{x}, \overline{a}))$ instead of $g(\phi(M^{\vert\overline{x}\vert}, \overline{a}))$.
\end{definition}

Note that definability always implies invariance. In an $\omega$-categorical context, by Ryll\-/Nardzewski, invariance implies definability and finiteness. 

Before introducing the notion of an $MS$-measurable structure, we follow the notation of \cite{WagDim} to speak of a  dimension function:

\begin{definition}
We call $\mathrm{dim : Def(M)}\to \mathbb{N}$ a \textbf{ dimension} if it satisfies the following conditions:
\begin{itemize}
    \item (Algebraicity) for $X$ finite and non-empty, $\mathrm{dim}(X)=0$;
    \item (Union) for $X, Y\in\mathrm{Def}(M)$, $\mathrm{dim}(X\cup Y)=\mathrm{max}\{\mathrm{dim}(X), \mathrm{dim}(Y)\}$; and
    \item (Additivity) for finite tuples $\overline{a},\overline{b}, \overline{c}$ from $M$, $\mathrm{dim}(\overline{a}\overline{b}/\overline{c})=\mathrm{dim}(\overline{a}/\overline{b}\overline{c})+\mathrm{dim}(\overline{b}/\overline{c})$.
\end{itemize}
\end{definition}

\begin{remark} In the additivity condition, by $\mathrm{dim}(\overline{a}/B)$ we mean $\mathrm{dim}(\mathrm{tp}(\overline{a}/B))$. For a partial type $\pi(\overline{x})$ over $B\subseteq M$, we define
\[\mathrm{dim}(\pi(\overline{x}))=\mathrm{min}\{\mathrm{dim}(\phi(M^{\vert\overline{x}\vert}, \overline{b}) \ \vert \ \pi(\overline{x})\vdash \phi(\overline{x}, \overline{b}) \}.\]
\end{remark}
\begin{remark} In the course of this paper, we work with definable dimensions in an $\omega$-categorical context. Hence, for $\overline{a}, \overline{b}$ finite tuples from $M$ we have that
\[\mathrm{dim}(\overline{a}/\overline{b})=\mathrm{dim}(\phi(M^{\vert\overline{x}\vert}, \overline{b})),\]
where $\phi(\overline{x}, \overline{b})$ is a formula isolating $\mathrm{tp}(\overline{a}/\overline{b})$.
\end{remark}

\begin{definition} We say that the function $\mu:\mathrm{Def}_{\overline{x}}(M)\to \mathbb{R}^{\geq 0}\cup \{\infty\}$ is a \textbf{measure} if it is a finitely additive function, i.e. for $X$ and $Y$ definable and disjoint in the same variable,
\[\mu(X\cup Y)=\mu(X)+\mu(Y).\]
We say $\mu$ is a \textbf{Keisler} measure if it takes values in $[0,1]$.
\end{definition}

We give here a slightly simplified definition of $MS$-measurability than the original following  \cite[Proposition 5.7]{MS}:

\begin{definition}\label{MScond} The $\mathcal{L}$-structure $\mathcal{M}$ is $\pmb{MS}$\textbf{-measurable} if there is a dimension-measure function $h:\mathrm{Def(M)}\to\mathbb{N}\times\mathbb{R}^{>0}$,  with notation $h(X)=\dimeas{X}$, satisfying the following conditions:
\begin{enumerate}
    \item The function $h$ is finite and definable.
    \item For $\overline{a}\in M^n$, $h(\{\overline{a}\})=(0,1)$.
    \item (Additivity) For $X, Y\subseteq M^n$ definable and disjoint,
     \begin{equation*} \mu(X\cup Y)=
   \left\{ \begin{array}{ll}
  \mu(X)+\mu(Y), & \text{for } \mathrm{d}(X)=\mathrm{d}(Y);\\
  \mu(Y) & \text{ for } \mathrm{d}(X)<\mathrm{d}(Y).
    \end{array}\right.
    \end{equation*}
    \item (Fubini) For $n\geq 2$, let $X\subseteq M^n$ be definable and for $1\leq m<n$, let $\pi:M^n\to M^m$ be a projection of $M^n$ to $m$-many coordinates. Suppose there is $(k, \nu)\in\mathbb{N}\times\mathbb{R}^{>0}$ such that for any $\overline{a}\in\pi(X)$, we have that $h(\pi^{-1}(\overline{a})\cap X)=(k, \nu)$. Then, $d(X)=d(\pi(X))+k$ and $\mu(X)=\mu(\pi(X))\nu$. 
\end{enumerate}
\end{definition}

\begin{remark}\label{eqMS}
From the definition it follows that being $MS$-measurable is a property of a theory, i.e. if $\mathcal{M}$ is $MS$-measurable, then so is any elementarily equivalent structure. We say that a many sorted structure $\mathcal{M}^*$ is $MS$-measurable if every restriction of $\mathcal{M}^*$ to finitely many sorts is $MS$-measurable. We have that if $\mathcal{M}$ is $MS$-measurable, then so is $\mathcal{M}^{eq}$ \cite[Proposition 5.10]{MS}.
\end{remark}

\begin{remark} An important, and somehow hidden, assumption of the definition of $MS$\-/measurability is that of \textbf{positivity}. That is, for $h(X)=(k, \nu)$, $\nu>0$. This will be vital to our proof of non-measurability since we will show that a given definable set must be assigned measure zero.
\end{remark}
\begin{remark} 
As noted in \cite[Proposition 5.3]{MS}, for $A$-definable $X$, we have an induced $A$-invariant probability measure on the definable subsets of $X$.
\end{remark}

\begin{remark} The dimension part of an $MS$-function is a definable dimension function, as defined above. Additivity follows from \cite[p.919]{WagDim}.
\end{remark}

In an $MS$-measurable context, for $\pi(\overline{x})$ a partial type over $B\subseteq M$, we define
\[\mu(\pi(\overline{x}))=\mathrm{inf}\{\mu(\phi(\overline{x}, \overline{b}))\vert \pi(\overline{x})\vdash \phi(\overline{x}, \overline{b}), \mathrm{d}(\phi(\overline{x}, \overline{b}))=\mathrm{d}(\pi(\overline{x}))\},\]
and we write $\mu(\overline{a}/B)$ for $\mu(\mathrm{tp}(\overline{a}/B))$.
For an $\omega$-categorical $MS$-measurable structure, $\mu(\overline{a}/\overline{b})\allowbreak=\mu(\phi(M^{\vert\overline{x}\vert}, \overline{b}))$, for any formula $\phi(\overline{x}, \overline{b})$ isolating $\mathrm{tp}(\overline{a}/\overline{b})$. 

\begin{remark}\label{Fub}
When speaking of the dimension and measure of types we shall also set the convention that $\mathrm{d}(\emptyset)=0$ and $\mu(\emptyset)=1$. This is helpful to treat the case of types over $\emptyset$ analogously to that of types over sets of parameters when dealing with Fubini. In fact, we shall make frequent use of the fact that, in an $\omega$-categorical $MS$-measurable structure, Fubini implies that \[h(\overline{a}\overline{b})=(\mathrm{dim}(\overline{a}\overline{b}), \mu(\overline{a}\overline{b}))=(\mathrm{dim}(\overline{a}/\overline{b})+\mathrm{dim}(\overline{b}),\mu(\overline{a}/\overline{b})\mu(\overline{b})).\]
 Hence, we adopt the conventions $\mathrm{d}(\emptyset)=0$ and $\mu(\emptyset)=1$ so that we may write 
  \[\mathrm{dim}(\overline{a})=\mathrm{dim}(\overline{a}/\emptyset)+\mathrm{dim}(\emptyset)=\mathrm{dim}(\overline{a}/\emptyset), \text{ and}\]
  \[\mu(\overline{a})=\mu(\overline{a}/\emptyset)\mu(\emptyset)=\mu(\overline{a}/\emptyset).\]
\end{remark}

\subsection{Dimension-independence}\label{dimind}

It is easy to obtain a notion of independence from an invariant dimension. In particular, in an $MS$-measurable structure, dimension-independence corresponds to non-forking independence. In \cite{EM}, this result is attributed to unpublished work of Ben-Yaacov. In this section we prove it briefly. In the next section, this will help us showing that in $\omega$-categorical $MS$-measurable structures we may take the dimension to be $SU$-rank without harm. 

\begin{definition} Let $\mathrm{d}:\mathrm{Dim}(\mathbb{M})\to \mathbb{N}$ be an invariant dimension, where $\mathbb{M}$ is a monster model. Let $\overline{a}$ be a tuple and $B, C$ be small subsets of $\mathbb{M}$. We say that $\overline{a}$ is $\mathrm{d}$-independent from $B$ over $C$, writing $\overline{a}\indep{C}{d} B$ if
\[\mathrm{d}(\overline{a}/BC)=\mathrm{d}(\overline{a}/C).\]
\end{definition}

The following is easy to prove from the basic properties of an invariant dimension:
\begin{prop} Let $\mathrm{d}:\mathrm{Dim}(\mathbb{M})\to \mathbb{N}$ be an invariant dimension. The relation of $\mathrm{d}$\-/independence is a notion of independence in the sense of \cite[Definition 4.1]{KPsimp}.
\end{prop}

Hence, in an $MS$-measurable context the dimension part of an $MS$-dimension-measure yields a notion of independence. We shall prove this 
coincides with non-forking independence. In order to prove this, first recall Lemma 3.5 from \cite{EM}:

\begin{lemma}\label{divlemma} Let $\mathcal {M}$ be $MS$-measurable. Let $X\subseteq M^{\vert\overline{x}\vert}$ be definable and $\phi(\overline{x}, \overline{y})$ be such that there is an indiscernible sequence $(\overline{b}_i)_{i<\omega}$ such that $\{\phi(\overline{x}, \overline{b}_i)\vert i<\omega\}$ is inconsistent and $\phi(M^{\vert\overline{x}\vert}, \overline{b}_i)\subseteq X$ for each $i<\omega$. Then, $\mathrm{d}(X)>\mathrm{d}(\phi(\overline{x}, \overline{b_i}))$.
\end{lemma}

From this result, it is easy to prove that dimension independence implies non-forking independence. 

\begin{prop}\label{indepimpl} Suppose that $\overline{a}\indep{C}{d} B$. Then, $\mathrm{tp}(\overline{a}/BC)$ does not divide over $C$. Hence, by simplicity, $\overline{a}\indep{C}{d} B$ implies $\overline{a}\indep{C}{f}B$, where $\indep{C}{f}$ is non-forking independence. 
\end{prop}
\begin{proof} Suppose by contrapositive that $\mathrm{tp}(\overline{a}/BC)$ divides over $C$. So there is $\phi(\overline{x}, \overline{d})\in\mathrm{tp}(\overline{a}/BC)$ such that $\{\phi(\overline{x}, \overline{d}_i)\vert i<\omega\}$ is $k$-inconsistent and $(\overline{d}_i)_{i<\omega}$ is $C$-indiscernible with $\overline{d}_0=\overline{d}$. Let $\chi(\overline{x})$ be a formula defined over $C$ witnessing $\mathrm{d}(\overline{a}/C)=k$. Consider $\phi(\overline{x}, \overline{d_i})\wedge\chi(\overline{x})$. Since $\phi(\overline{x}, \overline{d})\in\mathrm{tp}(\overline{a}/BC)$, we have that $\vDash\exists \overline{x}\phi(\overline{x}, \overline{d})\wedge \chi(\overline{x})$, and since $\overline{d}_i\equiv_C\overline{d}$, $\vDash\exists \overline{x}\phi(\overline{x}, \overline{d}_i)\wedge \chi(\overline{x})$. Thus, $\phi(\mathbb{M}^{\vert\overline{x}\vert}, \overline{d_i})\wedge\chi(\mathbb{M}^{\vert\overline{x}\vert})\subseteq \chi(\mathbb{M}^{\vert\overline{x}\vert})$, and $\{\phi(\overline{x}, \overline{d_i})\wedge\chi(\overline{x}) \vert i<\omega\}$ is inconsistent. The conditions of Lemma \ref{divlemma} are met and 
\[\mathrm{d}(\chi(\mathbb{M}^{\vert\overline{x}\vert})>\mathrm{d}(\phi(\mathbb{M}^{\vert\overline{x}\vert}, \overline{d_i})\wedge\chi(\mathbb{M}^{\vert\overline{x}\vert}))=\mathrm{d}(\phi(\mathbb{M}^{\vert\overline{x}\vert}, \overline{d})\wedge\chi(\mathbb{M}^{\vert\overline{x}\vert}))=\mathrm{d}(\phi(\mathbb{M}^{\vert\overline{x}\vert}, \overline{d})).\]
Where the second equality holds by invariance of the dimension and the last equality holds since $\phi(\mathbb{M}^{\vert\overline{x}\vert}, \overline{d})\subseteq \chi(\mathbb{M}^{\vert\overline{x}\vert})$. By our choice of $\chi(\overline{x})$, we have that $\mathrm{d}(\overline{a}/BC)<\mathrm{d}(\overline{a}/C)$, implying that $\overline{a}$ is not $\mathrm{d}$-independent from $B$ over $C$. Our claim then holds by contrapositive. 
\end{proof}

The other implication holds by the following result of Kim and Pillay \cite{KPsimp}[Theorem 4.2, Claim I]:
\begin{lemma} Let $T$ be an arbitrary theory and $\indep{}{*}$ be a notion of independence. Suppose that $a\indep{C}{f}B$. Then, $a\indep{C}{*}B$.
\end{lemma}

And this yields the desired result:

\begin{theorem}\label{dimtheom} Let $\mathcal{M}$ be an $MS$-measurable structure. Then, dimension independence is the same as non-forking independence.
\end{theorem}

\begin{remark} An alternative way to prove that dimension independence is the same of non-forking independence is by noting how the independence theorem is substantially proven in Theorem 2.18 of \cite{approxsubg}.
\end{remark}

\subsection{\texorpdfstring{$MS$}{MS}-measurability in an \texorpdfstring{$\omega$}{omega}-categorical context}\label{omegameas}

In this subsection, we introduce some basic equations which must be satisfied in $\omega$\-/categorical $MS$-measurable structures through Corollary \ref{amacor}. Observations regarding these equations allow us to prove that in an $\omega$-categorical context we may always take the dimension part of an $MS$-dimension-measure to be $SU$-rank, as shown in corollary \ref{SUok}.\\

The following result is a standard fact about $\omega$-categorical $MS$-measurable structures. As far as I can tell, a version of this is first proven in Elwes' PhD thesis  \cite[Lemma 5.2.1]{ElwesPhD}:

\begin{lemma} Let $\mathcal{M}$ be $\omega$-categorical and MS-measurable. Let $B\supseteq A$ be finite subsets of $M$. Suppose $p$ is a partial type over $A$. By $\omega$-categoricity, $p$ has finitely many complete extensions to $B$. Let $p_1, \dots, p_n$ be those which do not fork over $A$. Then, 
\begin{equation}
    \mu(p)=\sum_{i=1}^n \mu(p_i).
\end{equation}
\end{lemma}
\begin{proof}
Suppose the complete extensions of $p$ to $B$ are  $p_1, \dots, p_m$ with $m\geq n$ where $p_1, \dots, p_n$ have maximal dimension. By Theorem \ref{dimtheom}, these are the non-forking extensions of $p$ to $B$. Consider the sets
\[p_i(M^{\vert\overline{x}\vert}):=\left\{\overline{c}\in M^{\vert\overline{x}\vert} \big\vert \overline{c}\vDash p(\overline{x})\right\}.\]
By $\omega$-categoricity, these are disjoint definable sets such that $\bigcup_{i=1}^m p_i(M^{\vert\overline{x}\vert})=p(M^{\vert\overline{x}\vert})$. Hence, by additivity:
\[\mu(p)=\mu\left(\bigcup_{i=1}^n p_i(M^{\vert\overline{x}\vert})\cup \bigcup_{i>n}^m p_i(M^{\vert\overline{x}\vert})\right)=\mu\left(\bigcup_{i=1}^n p_i(M^{\vert\overline{x}\vert})\right)=\sum_{i=1}^n \mu(p_i).\]
Where the second last equality holds by additivity in the case $\mathrm{d}(X)<\mathrm{d}(Y)$, and the last equality due to additivity in the case of $\mathrm{d}(X)=\mathrm{d}(Y)$.
\end{proof}

By an easy application of Fubini (see Remark \ref{Fub}), we get equations in terms of types over the empty set and obtain:

\begin{corollary}\label{amacor} Let $\mathcal{M}$ be $\omega$-categorical and MS-measurable. Take $\overline{a}, \overline{b}, \overline{c}$ to be tuples from $M$. Let $\mathrm{tp}(\overline{c}_1/\overline{a}\overline{b}), \dots, \mathrm{tp}(\overline{c}_n/\overline{a}\overline{b})$ be the finitely many complete non-forking extensions of $\mathrm{tp}(\overline{c}/\overline{a})$ to $\overline{a}\overline{b}$. Then,
\begin{equation}\label{amacoreq}
    \frac{\mu(\overline{ac}) \mu( \overline{ab})}{\mu( \overline{a})}=\sum_{i=1}^n \mu(\overline{ab c_i}).
\end{equation}
\end{corollary}

We shall use the above corollary in our proof of non-measurability. There is a partial converse to this corollary. In fact, if an $\omega$-categorical structure has a function $\mu: S(\emptyset)\to\mathbb{R}^{>0}$ satisfying equations of the form of (\ref{amacoreq}) and an algebraicity condition, then it is $MS$-measurable. This can be extracted from the proof of Theorem \ref{changedim}. In this theorem, we show that in an $\omega$-categorical $MS$-measurable structure we may choose our dimension to be any definable dimension yielding non-forking independence as its notion of independence. 

\begin{theorem}\label{changedim} Suppose that $\mathcal{M}$ is an $\omega$-categorical $MS$-measurable structure with dimension-measure $h=(\mathrm{d}, \mu)$. Let $D:\mathrm{Def}(M)\to \mathbb{N}$ be any definable dimension for $\mathcal{M}$ whose induced notion of independence is non-forking independence. Then, $\mathcal{M}$ has a dimension-measure $h'=(D, \mu')$, where $\mu'$ agrees with $\mu$ on complete types over finite sets of parameters. 
\end{theorem}
\begin{proof} By Ryll-Nardzewski, $D$ is invariant and finite. The dimension-measure $h=(d, \mu)$ induces a function on types over finite sets of parameters $\mu^*:S_{\mathrm{fin}}(M)\to\mathbb{R}^{>0}$ given by $\mu^*(\overline{a}/\overline{b})=\mu(\overline{a}/\overline{b})$. We can write any set $X$ definable over $\overline{a}$ as a finite disjoint union of the sets of realisations of complete types $p_1, \dots, p_m$ over $\overline{a}$. Say that $p_1, \dots, p_n$ have maximal $D$-dimension. Then, we define,   
\[\mu'(X)=\sum_{i=1}^n \mu^*(p_i).\]
We need to show that our definition does not depend on the choice of parameters over which $X$ is defined. Suppose that $X$ is defined both over $\overline{a}$ and $\overline{b}$. We may assume without loss of generality that $\overline{a}\subseteq\overline{b}$. Each $p_i$ over $\overline{a}$ has finitely many complete extensions of maximal $D$-dimension to $\overline{b}$, say $p'_{i,1}, \dots, p'_{i, m_i}$, which, by assumption, are the non-forking extensions of $p_i$ to $\overline{b}$. By the dimension theorem \ref{dimtheom}, these are the extensions of $p_i$ of maximal $d$-dimension, and so, by $MS$-measurability, we know that
\[\mu(p_i)=\sum_{j=1}^{m_i}\mu(p'_{ij}).\]
But then, 
\[\sum_{i=1}^n\mu(p_i)=\sum_{i=1}^n\sum_{j=1}^{m_j}\mu(p'_{ij}),\]
where the equation on the left is the definition of $\mu'(X)$ as defined over $\overline{a}$ and the equation on the right is the definition of $\mu'$ for $X$ as defined over $\overline{b}$. This yields that $\mu'$ is well defined. Positivity, algebraicity, additivity and Fubini are easy to prove. Hence, $(D, \mu')$ yields an $MS$-dimension-measure for $\mathcal{M}$.
\end{proof}

An important consequence of this theorem is that for $MS$-measurable $\omega$-categorical structures we may take the dimension to be $SU$-rank. In a structure of finite $SU$-rank, $SU$-rank is a dimension function: it is additive by the Lascar inequalities \cite[Prop. 2.5.19]{Kimsimp}. In an $\omega$-categorical structure, it is definable being invariant. Since $MS$-measurable structures have finite $SU$-rank, in $\omega$-categorical $MS$-measurable structures, $SU$-rank yields a definable dimension. By definition it induces non-forking independence as its notion of independence. Hence, it satisfies all of the conditions of Theorem \ref{changedim}, and we get the following corollary.

\begin{corollary}\label{SUok} Suppose that the $\omega$-categorical structure $\mathcal{M}$ is $MS$-measurable with dimension-measure $h=(d, \mu)$. Then, $\mathcal{M}$ is also $MS$ measurable via the dimension-measure $h'=(SU, \mu')$, where $\mu'$ is as in Theorem \ref{changedim}. Hence, if an $\omega$-categorical structure $\mathcal{M}$ is $MS$-measurable, there is a dimension-measure on $\mathcal{M}$ where the dimension is given by $SU$-rank.
\end{corollary}

This statement is known to be false outside of an $\omega$-categorical context. In particular, in \cite[Remark 3.8]{EM}, Elwes and Macpherson show that there are $MS$-measurable structures where $SU$-rank is not definable. They also give an example of an $\omega$-categorical structure for which we can artificially choose the dimension in the $MS$-measure and $SU$-rank to differ. However, our theorem does prove that if there is an $MS$-measure in the $\omega$-categorical case, there is no harm in taking the dimension to be $SU$-rank.

Corollary \ref{SUok} is a powerful tool in the context of $MS$-measurable $\omega$-categorical structures. For example, Evans' proof \cite{Measam} employs a highly non-trivial theorem showing that in the class of $\omega$-categorical structures he is considering, any dimension function corresponds to a scaled $SU$-rank. Our result allows us to skip proving such a theorem. Indeed, for other $\omega$-categorical structures such a theorem is false (e.g. the example in \cite[Rem. 3.8]{EM}).

\section{Independence in measure}\label{indepmeas}
 In this section, we obtain another set of equations that hold for $MS$-measurable $\omega$-categorical structures. The main idea is that in an $MS$-measurable structure, non-forking independence induces probabilistic independence in the measure. This is shown explicitly in Theorem \ref{indeptheorem} and Corollary \ref{triang}, which yields equations that we will use later in our proof of non-$MS$-measurability. We thank Ehud Hrushovski for sharing his notes for \cite{AER}. In there, he proves a more general version of Theorem \ref{indeptheorem}, which implies the results in this section. Working in an $\omega$-categorical $MS$-measurable context greatly simplifies the tools and the amount of theory required to obtain the probabilistic independence theorem. Hence, we give proofs for these results in this section.\\
 
We work in an $\omega$-saturated $\mathcal{L}$-structure $\mathcal{M}$. Let $\mu:\mathrm{Def}_{\overline{y}}(M)\to [0,1]$ be an invariant measure and  $\phi(\overline{x}_1, \overline{y}), \psi(\overline{x}_2, \overline{y})$ be $\mathcal{L}$-formulas. We define the relation $R_\alpha(\overline{x}_1, \overline{x}_2)$ on $M^{\vert\overline{x}_1\vert}\times M^{\vert\overline{x}_2\vert}$ by 
\[R_\alpha(\overline{a}, \overline{b}) \text{ if and only if } \mu(\phi(\overline{a}, \overline{y})\wedge \psi(\overline{b}, \overline{y}))=\alpha.\]
From \cite[Prop.2.25]{approxsubg} we know that $R_\alpha(\overline{x}_1, \overline{x}_2)$ is stable.\\

We shall begin by showing that whether $R_\alpha(\overline{a}, \overline{b})$ holds of a pair of tuples only depends on the individual types of $\overline{a}$ and $\overline{b}$. This can be seen as a consequence of a commonly used corollary to the finite equivalence relation theorem \cite[Lemma 2.11]{Pillay} (see \cite{psfH} or \cite{KPsimp} for similar uses). We need to introduce some notation to express the result. 

\begin{definition} Let $\delta(\overline{x};\overline{z})$ be an $\mathcal{L}$-formula and $A$ a set of parameters. An instance of $\delta$ over $A$ is a formula $\delta(\overline{x};\overline{a})$ in $\mathcal{L}_A$, i.e. the language obtained by adding to $\mathcal{L}$ constants for the elements of $A$. A complete $\delta$-type over $A$ is a maximally consistent set of instances of $\delta$ over $A$. By $S_\delta^{\overline{x}} (A)$ we mean the set of $\delta$-types over $A$. By $FER_\delta(A)$ we denote the set of $A$-definable equivalence relations $E(\overline{x}, \overline{y})$ with finitely many classes such that for any $\overline{b}$, $E(\overline{x}, \overline{b})$ is elementary equivalent to a Boolean combination of instances of $\delta$ over $A$.
\end{definition}

With the above notation, we can prove the following well-know application of the finite equivalence relation theorem:

\begin{prop} Let $\mathcal{M}$ be an $\omega$-saturated $\mathcal{L}$-structure with $\mathrm{acl}^{eq}(\emptyset)=\mathrm{dcl}^{eq}(\emptyset)$. Let $\delta(\overline{x},\overline{y})$ be a stable $\mathcal{L}$-formula. Suppose that $\overline{a}$ and $\overline{b}$ are tuples from $\mathcal{M}$ with $\overline{a}\indep{}{}\overline{b}$. Then, whether $\vDash \delta(\overline{a},\overline{b})$ only depends on $\mathrm{tp}(\overline{a})$ and $\mathrm{tp}(\overline{b})$, but not on $\mathrm{tp}(\overline{ab})$.
\end{prop}
\begin{proof} Let $\overline{a}$ and $\overline{a}'$ be such that $\overline{a}\indep{}{}\overline{b},  \overline{a}'\indep{}{}\overline{b}$ and $\overline{a}\equiv\overline{a}'$. In particular, $\overline{a}$ and $\overline{a}'$ have the same $\delta$-type over $\emptyset$.\\

Let $B\subset M$ and suppose $q_1,q_2\in S^{\overline{x}}_\delta(B)$ are non-forking extensions of  $\mathrm{tp}_{\delta}(\overline{a})$, the $\sigma$ type of $\overline{a}$ over the empty set. By stability of $\delta(\overline{x}, \overline{y})$ and the finite equivalence relation theorem \cite[Lemma 2.11]{Pillay}, $q_1\neq q_2$ if and only if there is a finite equivalence relation $E(\overline{x}, \overline{y})\in\mathrm{FER}_\delta(\emptyset)$ such that $q_1(\overline{x})\cup q_2(\overline{y})\vdash E(\overline{x}, \overline{y})$. By $\mathrm{acl}^{eq}(\emptyset)=\mathrm{dcl}^{eq}(\emptyset)$,
$E(\overline{x}, \overline{y})$ is trivial, i.e. tuples with the same types  over the emptyset will be in the same equivalence classes. In particular, since for $\overline{a}$, $E(\overline{x}, \overline{a})$ is a Boolean combination instances of $\delta$ over over $\emptyset$, the equivalence class of $\overline{a}$ is entirely determined by $\mathrm{tp}_{\delta}(\overline{a})$, and so $q_1=q_2$. Hence, for any $B$ there is a unique $\delta$-type over $B$ which is a non-forking extension of $\mathrm{tp}_{\delta}(\overline{a})$. In particular, $\mathrm{tp}_{\delta}(\overline{a}/\overline{b})=\mathrm{tp}_{\delta}(\overline{a}'/\overline{b})$, which yields
\[\vDash \delta(\overline{a}, \overline{b})\leftrightarrow \delta(\overline{a}', \overline{b}).\]
We can apply the same reasoning in the variable $\overline{y}$ to prove that for $\overline{a}\indep{}{}\overline{b}$ and $\overline{a}'\indep{}{}\overline{b}'$ with $\overline{a}\equiv\overline{a}'$ and $\overline{b}\equiv\overline{b}'$ we have
\[\vDash \delta(\overline{a}, \overline{b})\leftrightarrow \delta(\overline{a}', \overline{b}').\]
\end{proof}

From this and stability of $R_\alpha(\overline{x}, \overline{y})$ we immediately get the following result for an invariant Keisler measure:

\begin{corollary}\label{constant} Let $\mu:\mathrm{Def}_{\overline{y}}(M)\to [0,1]$ be an invariant measure on an $\omega$-saturated $\mathcal{L}$-structure $\mathcal{M}$ with $\mathrm{acl}^{eq}(\emptyset)=\mathrm{dcl}^{eq}(\emptyset)$. Let $\phi_i(\overline{x}_i, \overline{y})$ for $i\in\{1,2\}$ be $\mathcal{L}$-formulas. Suppose that $\overline{a}_1, \overline{a}_2$ are tuples from $M$ such that $\overline{a}_1\indep{}{}\overline{a}_2$. Then, the value of $\mu(\phi_1(\overline{a}_1, \overline{y})\wedge \phi_2(\overline{a}_2, \overline{y}))$ only depends on $\mathrm{tp}(\overline{a}_1)$ and $\mathrm{tp}(\overline{a}_2)$, but not on $\mathrm{tp}(\overline{a}_1\overline{a}_2)$.

\end{corollary}

Let $\mathcal{M}$ be an $\mathcal{L}$ structure. Let $\overline{a}$ be a tuple from $M$ and $B\subseteq M$.  We write $\mathrm{loc}(\overline{a}/B)$ for the set of realisations of $\mathrm{tp}(\overline{a}/B)$ in $\mathcal{M}$.

As noted earlier, by Proposition 5.10 of \cite{MS}, if $\mathcal{M}$ is $MS$-measurable, then so is $\mathcal{M}^{eq}$. Let $\chi(x)$ be an $\mathcal{L}^{eq}(M^{eq})$-formula and let $\mathcal{M}^\star$ and $\mathcal{M}^\bullet$ be restrictions of $\mathcal{M}^{eq}$ to finitely many sorts containing the sorts of the variables and parameters in $\chi(x)$. Inspecting the proof of Proposition 5.10 we can see that the dimension measures $h^\star$ on $\mathcal{M}^\star$ and $h^\bullet$ on $\mathcal{M}^\bullet$ induced by the dimension-measure $h$ on $\mathcal{M}$ will agree on the value of the set defined by $\chi(x)$. Hence, we may write unambiguously
\[h(\chi(\mathcal{M}^{eq}))\]
for this value.

\begin{lemma}\label{a1a2form} Let $\mathcal{M}$ be $\omega$-categorical and $a_1, a_2$ be tuples from $M^{eq}$. Then, there is an $\mathcal{L}^{eq}$ formula $\psi(x_1, x_2)$ such that 

\begin{equation}\label{a1a2indep}
    \vDash \psi(a_1', a_2') \text{ if and only if } a_1'\equiv a_1 \text{ and } a_2'\equiv a_2 \text{ and } a_1'\indep{}{}a_2'.
\end{equation}
Moreover, if $\mathcal{M}$ is $MS$-measurable,
\[\mu\left(\psi(M^{eq}, M^{eq})\right)=\mu(a_1)\mu(a_2).\]
\end{lemma}
\begin{proof} The first part follows simply by Ryll-Nardzewski and invariance of non-forking independence. For the second part, we use Fubini. Firstly, since $\psi(x_1, x_2)$ isolates the types of its variables,
\[\vDash \exists x_1 \psi(x_1, b_2) \text{ if and only if } b_2\in\mathrm{loc}(a_2).\]
So, $\mu(\exists x_1 \psi(x_1, M^{eq}))=\mu(a_2)$. Secondly, consider $\mu(\psi(M^{eq}, b_2))$ for $b_2$ such that $\vDash \exists x_1 \psi(x_1, b_2)$. By invariance, $\mu(\psi(M^{eq}, b_2))=\mu(\psi(M^{eq}, a_2))$. But now, for $b_1\equiv a_1$, by the dimension theorem \ref{dimtheom},
\[\vDash \psi(b_1, a_2) \text{ if and only if } \mathrm{d}(b_1/a_2)=\mathrm{d}(b_1).\]
Hence, by additivity, $\mu(a_1)=\mu(\psi(M^{eq}, a_2))$. Considering Fubini applied to $\psi(x_1, x_2)$ yields the result. 
\end{proof}

\begin{fact}\label{indfact} If $\mathcal{M}$ is supersimple, $\omega$-categorical with $\mathrm{acl}^{eq}(\emptyset)=\mathrm{dcl}^{eq}(\emptyset)$, the independence theorem holds over the empty set.
\end{fact}
\begin{proof}
By $\omega$-categoricity, over finite sets, Lascar types and strong types coincide. By \cite[Proposition 3.2.12]{Kimsimp}, the independence theorem holds over $\mathrm{acl}^{eq}(\emptyset)$. Since $\mathrm{acl}^{eq}(\emptyset)=\mathrm{dcl}^{eq}(\emptyset)$, the independence theorem holds over the empty set. 
\end{proof}

At this point, we have the tools to show how in an $\omega$-categorical context, for a triplet of independent pairs, non-forking independence corresponds to probabilistic independence. 
 We state the following result in terms of $\mathcal{M}^{eq}$.

\begin{theorem}\label{indeptheorem} Let $\mathcal{M}^{eq}$ be $MS$-measurable, $\omega$-categorical with $\mathrm{acl}^{eq}(\emptyset)=\mathrm{dcl}^{eq}(\emptyset)$. 
Suppose that $a_1, {a}_2, {a}_3$ are tuples such that ${a}_i\indep{}{}{a}_j$ for $1\leq i<j\leq 3$. Then, $\mathrm{d}(\mathrm{tp}({a}_3/{a}_1)\cup \mathrm{tp}({a}_3/{a}_2))=\mathrm{d}(a_3)$ and
\begin{equation}\label{measeq0}
    \mu(\mathrm{tp}({a}_3/{a}_1)\cup \mathrm{tp}({a}_3/{a}_2))=\frac{\mu({a}_3/{a}_1)\mu({a}_3/{a}_2)}{\mu({a}_3)}. 
\end{equation}
\end{theorem}
In the theorem, the set defined by $\mathrm{tp}({a}_3/{a}_1)\cup \mathrm{tp}({a}_3/{a}_2)$ is $\mathrm{loc}({a}_3/{a}_1)\cap \mathrm{loc}({a}_3/{a}_2)$. The statement about the dimension corresponds to the independence theorem over $\emptyset$. But we get a much stronger result: the events $"x\in \mathrm{loc}(a_3/a_1)"$ and $"x\in \mathrm{loc}(a_3/a_1)"$ are probabilistically independent in the measure $\mu'$ induced by $\mu$ on definable subsets of $\mathrm{loc}(a_3)$.
\begin{proof}
We write $M$ instead of $M^{eq}$ in the proof to simplify the notation. For $i\in\{1, 2\}$, let $\phi_{i3}({x}_i, {x}_3)$ isolate $\mathrm{tp}({a}_i{a}_3)$. Let $\psi({x}_1, {x}_2)$ be the formula satisfying the condition \ref{a1a2indep} shown to exist in Lemma \ref{a1a2form}. Our proof consists of calculating in different ways the measure of the set $S$ defined by the formula
\[S(x_1, x_2, x_3):=\phi_{13}(x_1, x_3)\wedge \phi_{23}(x_2, x_3)\wedge \psi(x_1, x_2).\]
We apply Fubini, projecting onto the first two coordinates $\pi_{x_1x_2}:M^3\to M^2$. The formula $\exists x_3 S(x_1, x_2, x_3)$ is satisfied by any $b_1b_2\vDash \psi(x_1, x_2)$ by Fact \ref{indfact}. Hence,  
\[\exists x_3 S(x_1, x_2, x_3)\equiv \psi(x_1, x_2),\]
and so by Lemma \ref{a1a2form}, 
\begin{equation*}
    \mu(\pi_{x1x2}(S))=\mu(a_1)\mu(a_2).
\end{equation*}
For $b_1b_2\in M$ we consider $\pi_{x_1x_2}^{-1}(b_1b_2)\cap S$, i.e. the set defined by $S(b_1, b_2, x_3)$. This definable set is non-empty if and only if $b_1 b_2\vDash \psi(x_1, x_2)$. Also, by Fact \ref{indfact}, for any such pair, $S(b_1, b_2, x_3)$ will have maximal dimension. Finally, by Corollary \ref{constant}, 
\[\mu(S(b_1, b_2, M))=\mu(S(a_1, a_2, M))= \mu(\mathrm{tp}({a}_3/{a}_1)\cup \mathrm{tp}({a}_3/{a}_2)).\]
Hence, by Fubini with respect to $\pi_{x_1x_2}$ we obtain:
\begin{equation}\label{measeq1}
    \mu(S)=\mu(\mathrm{tp}({a}_3/{a}_1)\cup \mathrm{tp}({a}_3/{a}_2))\mu(a_1)\mu(a_2).
\end{equation}
Let us compute the measure with respect to $\pi_{x_3}:M^3\to M$. The formula $\exists x_1 x_2 S(x_1, x_2, x_3)$ isolates $\mathrm{tp}(a_3)$ and so $\mu(\exists x_1 x_2 S(x_1, x_2, M))=\mu(a_3)$. Meanwhile, for $b_3\in M$, the formula $S(x_1, x_2, b_3)$ is consistent only if $b_3\equiv a_3$. By invariance we may consider $S(x_1, x_2, a_3)$. To compute the measure of the set defined by this formula we apply Fubini again. Projecting onto $x_1$, $\exists x_1 S(x_1, x_2, a_3)$ isolates $\mathrm{tp}(a_2/a_3)$ and so, $\mu(\exists x_1 S(x_1, M, a_3))=\mu(a_2/a_3)$. Meanwhile, for $b_2\equiv_{a_3} a_2$ consider $S(x_1, b_2, a_3)$. Again, by invariance this defines a set with the same measure as $S(x_1, a_2, a_3)$. We claim that $\mu(S(M, a_2, a_3))=\mu(a_1/a_3)$. By Lemma \ref{a1a2form}, there is a formula $\chi(x_1, x_2)$ over $a_3$ isolating the pairs $b_1\equiv_{a_3} a_1$ and $b_2\equiv_{a_3} a_2$ which are independent over $a_3$. So, we may consider $S(M, a_2, a_3)$ as the disjoint union
\[(S(M, a_2, a_3)\wedge \chi(M, a_2)) \sqcup (S(M, a_2, a_3)\wedge \neg\chi(M, a_2)).\]
But the definable set on the right of the disjoint union has lower dimension by the dimension theorem, and so, by additivity,
\[\mu(S(M, a_2, a_3))=\mu(S(M, a_2, a_3)\wedge \chi(M, a_2)).\]
Consider the extension of $\mathrm{tp}(a_1/a_3)$ to $a_2$. By the same reasoning, 
\[\mu(a_1/a_3)=\mu(\phi_{13}(M, a_3)\wedge \chi(M, a_2)).\]
Noting that $\chi(x_1, x_2)$ implies $\psi(x_1, x_2)$, we get that 
\[\mu(S(M, a_2, a_3))=\mu(a_1/a_3).\]
Hence, our calculations with respect to $\pi_{x_3}$ yield:
\begin{equation}\label{measeq2}
    \mu(S)=\mu(a_3)\mu(a_2/a_3)\mu(a_1/a_3).
\end{equation}

Recall that by Fubini $\mu(a_i a_j)=\mu(a_i/a_j)\mu(a_j)$. Using this fact and comparing the equations \ref{measeq1} and \ref{measeq2}, we obtain 
\begin{equation}\label{measeq5}
 \mu(\mathrm{tp}({a}_3/{a}_1)\cup \mathrm{tp}({a}_3/{a}_2))=  \frac{\mu({a}_2{a}_3)\mu({a}_1{a}_3)}{\mu({a}_1)\mu({a}_2)\mu({a}_3)}.
\end{equation}
Applying Fubini again, we get the desired result.
\end{proof}

\begin{corollary}\label{triang}
Let $\mathcal{M}^{eq}$ be $MS$-measurable, $\omega$-categorical with $\mathrm{acl}^{eq}(\emptyset)=\mathrm{dcl}^{eq}(\emptyset)$. Suppose that ${a}_1, a_2, {a}_3$ are tuples such that ${a}_i\indep{}{}{a}_j$ for $1\leq i<j\leq 3$. Let $p_{ij}({x}_i, {x}_j)=\mathrm{tp}({a}_i, {a}_j)$. Then, 
\[\mathrm{d}(p_{12}({x}_1, {x}_2)\cup p_{23}({x}_2, {x}_3)\cup p_{13}({x}_1, {x}_3))=\sum_{i\leq 3} \mathrm{d}(a_i), \text{ and }\]
\[\mu(p_{12}({x}_1, {x}_2)\cup p_{23}({x}_2, {x}_3)\cup p_{13}({x}_1, {x}_3))=\frac{\mu({a}_1 {a}_2 )\mu({a}_1 {a}_3 )\mu({a}_2 {a}_3 )}{\mu({a}_1)\mu({a}_2)\mu({a}_3)}.\]
\end{corollary}
\begin{proof}
Considering the set defined by $p_{12}({x}_1, {x}_2)\cup p_{23}({x}_2, {x}_3)\cup p_{13}({x}_1, {x}_3)$, we apply Fubini by projecting onto the first two coordinates. The projection isolates 
$\mathrm{tp}(a_1a_2)$, and the fiber of $a_1a_2$ isolates $\mathrm{tp}({a}_3/{a}_1)\cup \mathrm{tp}({a}_3/{a}_2)$. By Fubini
\[\mu(p_{12}({x}_1, {x}_2)\cup p_{23}({x}_2, {x}_3)\cup p_{13}({x}_1, {x}_3))=   \mu(\mathrm{tp}({a}_3/{a}_1)\cup \mathrm{tp}({a}_3/{a}_2))\mu({a}_1{a}_2).\]
Substituting with equation \ref{measeq5}, the result follows. 
\end{proof}
\section{Hrushovski constructions}\label{Hrushconstr}

In this section, we focus on $\omega$-categorical Hrushovski constructions. We begin with Subsection \ref{introH} which gives a brief introduction to these structures following \cite{Evans} and \cite{Wagner:ST}. Then, in Subsection \ref{Hrusheqns}, we find the form that the equations of Corollary \ref{amacor} would take in the context of $\omega$-categorical Hrushovski constructions if these were $MS$-measurable.

\subsection{Building \texorpdfstring{$\omega$}{omega}-categorical Hrushovski constructions}\label{introH}
We focus on graphs for simplicity of notation, but we can build Hrushovski constructions in relational languages in general \cite{Wagner:ST}. Let $\overline{\mathcal{K}}$ be the class of graphs and $\mathcal{K}$ be the class of finite graphs. Now, for $\alpha\in\mathbb{R}^{>0}$, and $A$ a finite graph, we define the \textbf{predimension} of $A$, $\delta(A)$ to be:
\[\delta(A)=\alpha\vert  A\vert -\vert  E(A)\vert ,\]
where $\vert  E(A)\vert $ is the number of edges of $A$.

For $A\subseteq B,C\in\mathcal{K}$, suppose $g:A\to B$ and $f:A\to C$ are embeddings of $A$ in $B$ and $C$ respectively. The \textbf{free amalgamation} of $B$ and $C$ over $A$, is the (unique) graph $B\amalg_A C$ such that there are embeddings $h:B\to B\amalg_A C$, and $k:C\to B\amalg_A C$ such that $h(g(A))=k(f(A))$ and $E(B\amalg_A C)=E(B)\cup E(C)$.

The predimension we introduced satisfies \textbf{submodularity}. That is, for $D\in\overline{\mathcal{K}}$, and $B,C\subset D$ finite, we have that
\[\delta(B\cup C)\leq \delta(B)+\delta(C)-\delta(B\cap C),\]
with equality holding if and only $B$ and $C$ are freely amalgamated over $B\cap C$.


\begin{definition} Let $B\in\overline{\mathcal{K}}$, $A\subseteq B$ finite. We say that $A$ is \textbf{d-closed} in $B$, and write $A\leq B$ if $\delta(A)<\delta(B')$ for any finite $B'$ such that $A\subsetneq B'\subseteq B$. We say that $A$ is \textbf{self sufficient} in $B$ and write $A\leq^*B$ if $\delta(A)\leq\delta(B')$ for any finite $B'$ such that $A\subseteq B'\subseteq B$.
\end{definition}

Let $f:\mathbb{R}^{\geq0}\to\mathbb{R}^{\geq0}$ be continuous, increasing and unbounded with $f(0)=0$. We let 
\[\mathcal{K}_f:=\{A\in\mathcal{K}: \delta(A')\geq f(\vert A'\vert ) \text{ for any }A'\subseteq A\}.\]

\begin{definition} 
We say that $\mathcal{K}_f$ has the \textbf{free amalgamation property} if given $A_0\leq A_1, A_2\in\mathcal{K}_f$ we have that $A_1\amalg_{A_0} A_2\in\mathcal{K}_f$.
\end{definition}

When $\mathcal{K}_f$ has the free amalgamation property we can build an $\omega$-categorical Hrushovski construction from it \cite[Theorems 3.2 \& 3.19]{Evans}:

\begin{theorem}\label{Hrushconstrtheorem} Suppose that $\mathcal{K}_f$ has the free amalgamation property. Then, there is a unique countable structure $\mathcal{M}_f$ such that:
\begin{itemize}
    \item \textbf{Union of Chains:} $\mathcal{M}_f$ is given by the union of a chain $M_1\leq M_2\leq \dots$ where $M_i\in\mathcal{K}_f$.
    \item \textbf{Substructures in} $\pmb{\mathcal{K}_f}$\textbf{:} every $A\in\mathcal{K}_f$ is isomorphic to some $A'\leq M_f$.
    \item \textbf{Extension Property:} given $A\leq M_f$ finite and $A\leq B\in \mathcal{K}_f$, there is an embedding of $B$ over $A$ in $\mathcal{M}_f$ such that $B\leq M_f$.
\end{itemize}
Furthermore, $\mathcal{M}_f$ is $\omega$-categorical.
\end{theorem}

We say that $f$ is a \textbf{good function} if it is piece-wise smooth, $f'(x)\leq 1/x$ and $f'(x)$ is non-decreasing.

\begin{lemma}\label{showfreeam} Suppose that $f(x)$ is a good function for $x\geq t$. and that for $A_0\leq A_1,A_2$, with $\vert   A_i\vert \leq t$, we have $\delta(A_0\amalg_{A_1}A_2)\geq f(\vert A_0\amalg_{A_1}A_2\vert )$. Then $\mathcal{K}_f$ has the free amalgamation property.
\end{lemma}
\begin{proof}
This is substantially the proof of Example 3.20 in \cite{Evans}, which goes back to Hrushovski's original unpublished note \cite{omps}. 
\end{proof}

Furthermore, when $\alpha\in\mathbb{Q}$, choosing $f$ with slow enough growth ensures that $\mathcal{M}_f$ is supersimple of finite $SU$-rank by \cite[Theorem 3.6]{supersimple} (originally proved in \cite{Udiamalg}). This is explained further in the Appendix. 

For $B\in\mathcal{K}_f$ or $B=\mathcal{M}_f$ and finite $A\subseteq B$ there is a unique smallest d-closed subset of $B$ containing $A$ which we call $\mathrm{cl}_B(A)$, i.e, the \textbf{closure} of $A$ in $B$ \cite[Lemma 2.1]{Wagner:ST}.

In the context of $\mathcal{M}_f$, $\mathrm{cl}_{M_f}(A)$ corresponds to algebraic closure of $A$ in $\mathcal{M}_f$ and is bounded above by $f^{-1}(\alpha\vert A\vert )$ \cite[Theorem 3.19]{Evans}. In general, we shall avoid the subscript when it is clear we are speaking of $\mathrm{cl}_{M_f}(A)$.

The predimension of a closed set induces a natural notion of \textbf{dimension}. For $D$ finite or $D=\mathcal{M}_f$, given $A\subseteq D$, we write $\mathrm{d}_D(A)$ for $\delta(\mathrm{cl}_D(A))$. If $\alpha\in\mathbb{N}$, then this dimension is also a natural number, and it may be re-scaled to be a natural number if $\alpha\in\mathbb{Q}$. As above, we omit the subscript and write $d(A)$ for $d_{\mathcal{M}_f}(A)$. 

We write $d_D(A/B)$ for $d_D(AB)-d_B(B)$, where we write $AB$ for $A\cup B$. We say that $A$ and $B$ are \textbf{independent} over $C$ in $D$ if $d_D(A/BC)=d_D(A/B)$. This implies that $\mathrm{cl}_D(AB)\cap \mathrm{cl}_D(AC)=\mathrm{cl}_D(A)$. When $\mathcal{M}_f$ is supersimple (of finite $SU$ rank) we have that this notion of independence induced by the Hrushovski dimension in $\mathcal{M}_f$ is precisely non-forking independence and the Hrushovski dimension on $\mathcal{M}_f$ is $SU$-rank. \\

We have a clear picture of what independence looks like in $\mathcal{M}_f$ \cite[Claim 6.2.9]{Kimsimp}:

\begin{lemma}\label{indepfree} Let $D\in\mathcal{K}_f$ or $D=\mathcal{M}_f$. Suppose $A=B\cap C\leq B, C\leq D$. Then, $B$ and $C$ are independent over $A$ in $D$ if and only if $B$ and $C$ are freely amalgamated over $A$ and $\delta(B\cup C)=\delta(\mathrm{cl}_D(B\cup C))$.
\end{lemma}

\begin{remark}\label{moreind}
With the notation above, if $B$ and $C$ are independent over $A$ in $D$, $BC\leq^* \mathrm{cl}_D(B\cup C)$. To see this, suppose by contradiction that there is some $F$ such that  $BC\subseteq F\subsetneq \mathrm{cl}_D(B\cup C)$ and $\delta(F)<\delta(BC)$. We may choose $F$ to be maximal so that for $F\subsetneq H\subseteq \mathrm{cl}_D(B\cup C)$, $\delta(H)>\delta(F)$. Then,
\[BC\subseteq F\leq \mathrm{cl}_D(B\cup C)\leq D.\] But this contradicts the definition of closure in $D$ as the minimal d-closed subset of $D$ containing $BC$.
\end{remark}

\subsection{MS-measurable \texorpdfstring{$\omega$}{omega}-categorical Hrushovski constructions}\label{Hrusheqns}

In this section, we consider some equations which would hold in an $MS$-measurable $\omega$-categorical Hrushovski construction. Suppose $A\subseteq M_f$, and $B=\mathrm{cl}_{\mathcal{M}_f}(A)$. Let $\overline{a}\overline{b}$ be an enumeration of $B$ where $\overline{a}$ is an enumeration of $A$. By the extension property in Theorem \ref{Hrushconstrtheorem}, the quantifier-free type of $\overline{a}\overline{b}$ determines the type of of $\overline{a}$ in $\mathcal{M}_f$ \cite[Corollary 2.4]{Wagner:ST}. For $A\leq M_f$, we write $\theta_{\overline{a}}(\overline{x})$ for the formula isolating $\mathrm{tp}(\overline{a})$.

\begin{notation}\label{notameas} Note that by Fubini, any $MS$-dimension-measure $h$ is invariant under permutation of variables. That is, for $\overline{x}_\sigma$ a permutation of the variables of $\overline{x}$, and $\phi_\sigma(\overline{x})=\phi(\overline{x}_\sigma)$, 
\[h(\phi_\sigma(M^{\vert\overline{x}\vert}))=h(\phi(M^{\vert\overline{x}\vert})).\]
Hence, if $\mathcal{M}_f$ is $MS$-measurable, for $A\leq M_f$, we may write $\mu(A):=\mu(\theta_{\overline{a}}(\overline{x}))$ without ambiguity. For $A\in\mathcal{K}_f$ we write $\mu(A)$ for $\mu(\theta_{\overline{a}}(\overline{x}))$, where $\overline{a}$ is an enumeration of a copy $A'$ of $A$ in $M_f$ such that $A'\leq M_f$.
\end{notation}

By the notation $\mathrm{tp}(A/B)$ we mean $\mathrm{tp}(\overline{a}/B)$ where $\overline{a}$ is an enumeration of $A$.

\begin{theorem}\label{Hrush}
Let $\mathcal{M}_f$ be the $\omega$-categorical Hrushovski construction for the class $\mathcal{K}_f$. Suppose that $\mathcal{M}_f$ is MS-measurable. Let $A\leq B, C\leq M_f$ be finite. Let $\mathrm{tp}(C_1/B), \dots, \mathrm{tp}(C_n/B)$ be the finitely many non-forking extensions of $\mathrm{tp}(C/A)$ to $B$. Let $D_i$ be $\mathrm{cl}(BC_i)$. Then, we have that:
\begin{equation}\label{Heq}
    \frac{\mu(B)\mu(C)}{\mu(A)}=\sum_{i=1}^n \frac{\mu(D_i)}{\vert\mathrm{Aut}(D_i/BC_i)\vert},
\end{equation}
where $\vert\mathrm{Aut}(D_i/BC_i)\vert$ is the number of automorphisms of $D_i$ fixing $BC_i$ pointwise. 
\end{theorem}
\begin{proof} Let $\overline{a}\overline{b}$ be an enumeration of $B$ where $\overline{a}$ is an enumeration of $A$. Similarly, let $\overline{a}\overline{c}$ be an enumeration of $C$. Let $\overline{a}\overline{b}\overline{c}_i\overline{d}_i$ be an enumeration of $D_i$ where, $\overline{a}\overline{b}\overline{c}_i$ is the corresponding enumeration of $ABC_i$. By Corollary \ref{amacor}, 
\begin{equation}
    \frac{\mu(B)\mu(C)}{\mu(A)}=\sum_{i=1}^n \mu(\mathrm{tp}(\overline{a}\overline{b}\overline{c}_i))
\end{equation}
We know that $\mathrm{tp}(\overline{a}\overline{b}\overline{c}_i)$ is isolated by $\exists\overline{w}\theta_{\overline{a}\overline{b}\overline{c}_i\overline{d}_i}(\overline{x}\overline{y}\overline{z}\overline{w})$. Note that $\vert\{e\in M_f \vert \theta_{\overline{a}\overline{b}\overline{c}_i\overline{d}_i}(\overline{a}\overline{b}\overline{c_i}\overline{e})\}\vert$ is finite and counts the number of automorphisms of $D_i$ fixing $BC_i$ pointwise. Hence, by Fubini,
\[\mu(\mathrm{tp}(\overline{a}\overline{b}\overline{c}_i)) \vert\mathrm{Aut}(D_i/BC_i)\vert=\mu(D_i).\]
This yields the desired equation.
\end{proof}

Following Lemma \ref{indepfree}, we can express the conditions for a graph $D_i$ to appear as $\mathrm{cl}(BC_i)$ in the theorem above in terms of $\mathcal{K}_f$.

\begin{definition} Let $A_0\leq A_1, A_2$, $A_1\amalg_{A_0} A_2\subseteq D\in \mathcal{K}_f$. We say $D$ is an \textbf{eventual closure} of $A_1\amalg_{A_0} A_2$ if $A_1$ and $A_2$ are freely amalgamated over $A_0$ in $D$, $A_1\cup A_2\leq^*D$, $\delta(A_1\cup A_2)=\delta(D)$, and $A_i\leq D$ for each $i\in\{0,1,2\}$.
\end{definition}

\begin{remark} Consider the closures $D_i$ in Theorem \ref{Hrush}. We have that there is $C_i\leq M_f$ independent from $B$ over $A$ with $\mathrm{tp}(C_i/A)=\mathrm{tp}(C/A)$  and closure $D_i$ if and only if there is $D_i'\in \mathcal{K}_f$, isomorphic to $D_i$ as a graph, which is an eventual closure of $B\amalg_A C$. The left-to-right implication follows by Lemma \ref{indepfree} and Remark \ref{moreind}. The right-to-left implication follows by the extension property, noting that $B\leq D_i'\in\mathcal{K}_f$.
\end{remark}
 Hence, we may rephrase Theorem \ref{Hrush} as follows:

\begin{corollary}\label{Hrushcor} Let $\mathcal{M}_f$ be an $\omega$-categorical Hrushovski construction with amalgamation class $\mathcal{K}_f$. Suppose $\mathcal{M}_f$ is $MS$-measurable. Let $A\leq B, C\in\mathcal{K}_f$. Let $D_1, \dots D_n$ be the eventual closures of $B\amalg_A C$. Then, 
\begin{equation*}
    \frac{\mu(B)\mu(C)}{\mu(A)}=\sum_{i=1}^n \frac{\mu(D_i)}{\vert\mathrm{Aut}(D_i/BC)\vert}.
\end{equation*}
\end{corollary}
\section{A non-measurability proof}\label{finally}

In this section, we give a class of $\omega$-categorical Hrushovski constructions. These are supersimple of finite $SU$-rank, as proven in the Appendix \ref{appendix}. However, we will prove in Subsection \ref{proof!} that these are not $MS$-measurable. In Subsection \ref{evcl}, we introduce these structures and show how we have control on small enough eventual closures in them. 
\subsection{The structure \texorpdfstring{$\mathcal{M}_f$}{Mf} and small eventual closures}\label{evcl}

We begin by giving the class of structures that we work with.

\begin{construction}\label{actualf} We set $\alpha=2$. Let $f:\mathbb{R}^{>0}\to \mathbb{R}^{>0}$ be such that for $t\geq 6$, $f$ is a good function with $f(3\cdot t)\leq f(t)+1$ and for $t\leq 6$, $f$ is piece-wise linear with $f(1)=2, f(4)=5, f(6)=6$. An illustration of the initial values of $f$ is given by Figure \ref{fig:newf}. As proven in the Appendix \ref{appendix}, the class $\mathcal{K}_f$ has the free amalgamation property, and so by Theorem \ref{Hrushconstrtheorem} we can build an $\omega$-categorical Hrushovski construction $\mathcal{M}_f$. 
\end{construction}

An example of a function satisfying the conditions of Construction \ref{actualf} is the function taking the values specified for $t\leq 6$ and corresponding to $\log_3(x)-\log_3(6)+6$ for $t>6$. Of course, there are uncountably many functions satisfying the conditions of our construction.

From now on, by $\mathcal{M}_f$ we mean an $\omega$-categorical Hrushovski construction obtained following Construction \ref{actualf}.
\begin{figure}
    \centering
 \includegraphics[width=\textwidth]{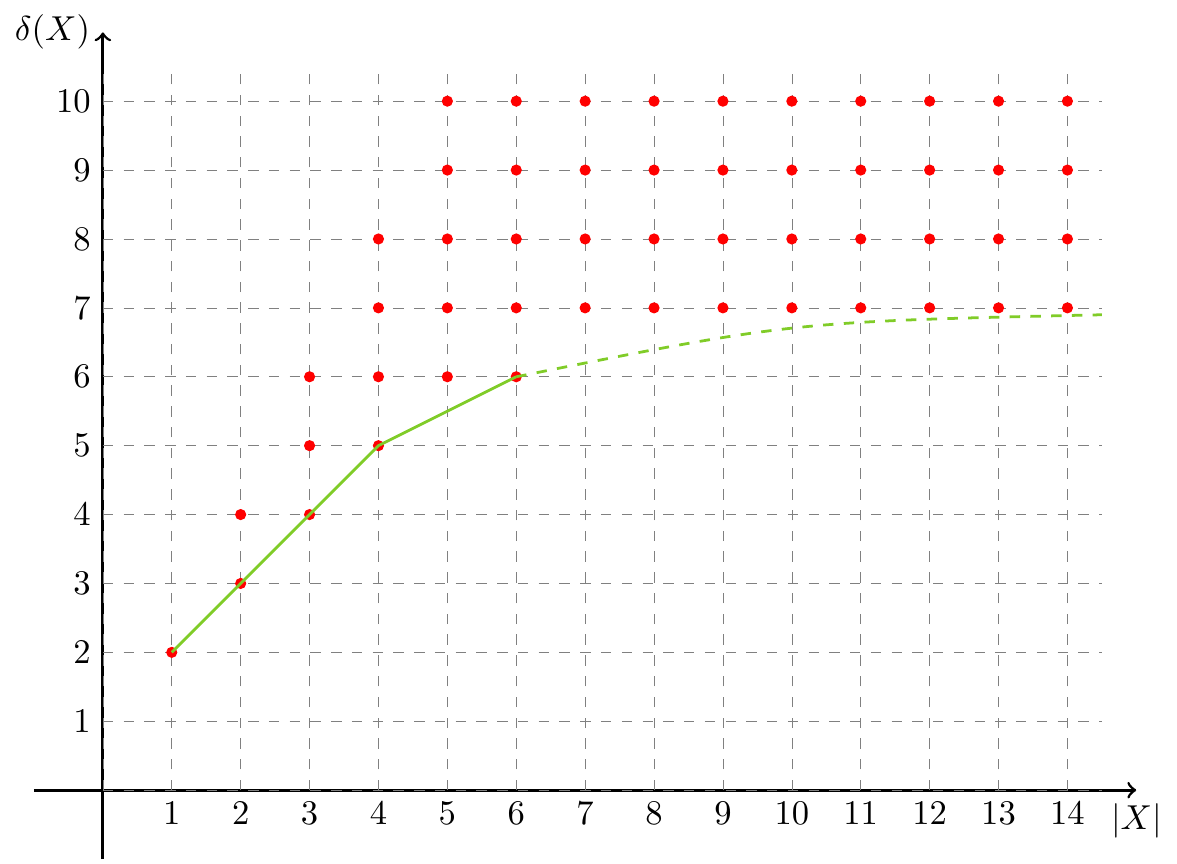}
\caption{The initial values of $f$ are represented by the line in the figure. The dots represent positions $(\vert X\vert , \delta(X))$ of the graphs $X$ in $\mathcal{K}_f$ (some of the dots for $\vert X\vert \geq 5$ lie outside of the figure). Note that $f(1)=2$ and $f(2)=3$ are necessary to include vertices and edges in $\mathcal{K}_f$. $f(3)=4, f(4)=5, f(6)=6$ imply that the smallest cycle in $\mathcal{K}_f$ is a $6$-cycle. We also have that vertices and edges are always closed in $\mathcal{M}_f$.}
\label{fig:newf}
\end{figure}

\begin{lemma}[Basic properties of $\mathcal{M}_f$] The $\omega$-categorical Hrushovski construction $\mathcal{M}_f$ is supersimple of $SU$-rank $2$. Furthermore, it has weak elimination of imaginaries. Finally, we may choose $f$ with growth slow enough that $\mathcal{M}_f$ satisfies independent $n$ amalgamation over finite algebraically closed sets for any $n\geq 3$ or even for all $n\in\mathbb{N}$.
\end{lemma}
\begin{proof} See Appendix \ref{appendix}.\end{proof}

As we shall see, the choice of values of $f(t)$ for $t\leq6$ ensures that the eventual closures of graphs of small enough dimension are easy to enumerate. 

\begin{definition} Let $A\in\mathcal{K}_f$. Suppose $A\leq^* A'\in\mathcal{K}_f, \delta(A')=\delta(A)$, and $\vert A'\setminus A\vert =1$. Then, we call $A'$ a \textbf{one point extension} of $A$.
\end{definition}

Since $\alpha=2$, a one-point extension $A'$ of $A$ will be obtained by adding to $A$ a single vertex $v$ joined by two edges to $A$.

\begin{lemma}\label{iterate1ext} Let $A, B\in\mathcal{K}_f$, $A\leq^*B, \delta(A)=\delta(B),$ and $\vert B\setminus A\vert <6$. Then, $B$ is obtained from $A$ by iterating one-point extensions.
\end{lemma}

\begin{proof}
Note that with our choice of $f$ for $X\in \mathcal{K}_f$ such that $\vert X\vert <6$, $\vert E(X)\vert <\vert X\vert $. In fact, for $\vert X\vert <6$, $f(X)>\vert X\vert $. Since $X\in\mathcal{K}_f$, 
\begin{align}
    & 2\vert X\vert -\vert E(X)\vert =\delta(X)\geq f(\vert X\vert )>\vert X\vert  \\
   \Rightarrow & \vert E(X)\vert <\vert X\vert . \label{e<x}
\end{align}
Now, consider the graph $X$ induced by the vertices in $B\setminus A$. Since $\vert X\vert <6$, $\vert E(X)\vert <\vert X\vert $. Let $E(A;X):=\{\{u, v\}\in E(B) \vert  u\in A, v\in X\}$. By $\delta(B)=\delta(A)$, we have that:
\[\delta(B)=2\vert B\vert -E(B)=2\vert A\vert +2\vert X\vert -\vert E(A)\vert -\vert E(X)\vert -\vert E(A;X)\vert =2\vert A\vert -\vert E(A)\vert =\delta(A),\]
and so since $\vert E(X)\vert <\vert X\vert $, 
\[0=2\vert X\vert -\vert E(X)\vert -\vert E(A;X)\vert >\vert X\vert -\vert E(A;X)\vert ,\]
which yields $\vert X\vert <\vert E(A;X)\vert $. By the pigeonhole principle there must be a vertex $v$ of $X$ connected to $A$ by two or more points. Note that if $v$ were connected by three or more points $\delta(A\cup\{v\})<\delta(A)$, and so $A\not\leq^*B$. Hence, $A\cup\{v\}$ must be a one-point extension of $A$.
The same argument can be iterated for $A\cup\{v\}$ so that we see $B$ is constructed by iterating one-point extensions.
\end{proof}

\begin{remark}\label{1ptevc}
Let $D$ be a proper eventual closure of $A_1\amalg_{A_0}A_2$ with $\vert D\setminus A_1\cup A_2\vert <6$. Then, as noted in Lemma \ref{iterate1ext}, $D$ is obtained by iterating one-point extensions $D_1, \dots, D_k$ for $k\leq 5$. One can see that each $D_i$ is also an eventual closure.

We shall call an eventual closure of $A_1\amalg_{A_0}A_2$ which is also a one-point extension a \textbf{one point closure}.
\end{remark}

Given Lemma \ref{iterate1ext} and Corollary \ref{Hrushcor}, we get:

\begin{corollary}\label{noextcor} Let $A\leq\mathcal{M}_f$, and $A\leq B, C\in\mathcal{K}_f$. Let $N=\lfloor f^{-1}(\delta(B\amalg_A C))\rfloor$ so that $N$ is the largest size of a graph of predimension $\delta(B\amalg_A C)$. Suppose that $N-\vert B\amalg_A C\vert <6$ and that the maximal distance between a point in $B\setminus A$ and a point in $C\setminus A$ in $B\amalg_A C$ is $\leq 3$. Then, there is no proper eventual closure for $B\amalg_A C$, and so, 
\[\mu(B\amalg_A C)=\frac{\mu(B)\mu(C)}{\mu(A)}.\]
\end{corollary}
\begin{proof} Since no graph in $\mathcal{K}_f$ contains cycles of length $<6$, $B\amalg_A C$ has no proper eventual closures. In fact, in a one point closure of $B\amalg_A C$ by a point $v$, $v$ must be attached to a vertex in $B\setminus A$ and a vertex in $C\setminus A$ since $B$ and $C$ are closed in the eventual closure. But then, this one-point closure would contain a $k$-cycle for $k<6$. 
\end{proof}

\subsection{Proving non-measurability}\label{proof!}

We are ready to prove that the structures built in Construction \ref{actualf} are not $MS$-measurable. We first find some equations that an $MS$-measure should satisfy and then note that these equations imply that some definable set has measure zero, breaking the positivity condition of $MS$-measures.

\begin{notation}
We still follow the notation of \ref{notameas}. However, for clarity, we represent graphs pictorially. For example, for $A\in\mathcal{K}_f$ a path of length two, instead of writing $\mu(A)$, we may write $\mu\left(\ptwo\right)$, so that it is clear from our notation which graph we are talking about. Note that when we speak of "the measure of $A$", we actually mean "the measure of the definable subset of $M_f^{\vert A\vert }$ consisting of copies of $A$ which are such that $A\leq M_f$". By our choice of $f$, vertices and edges isolate complete types in $\mathcal{M}_f$.
\end{notation}

\begin{prop}\label{manyeqns} Suppose that $\mathcal{M}_f$ is $MS$-measurable with dimension given by $SU$-rank. Assume without loss of generality that the measure of a single vertex is $1$. Let $\lambda$ be the measure of an edge. Then, $\mu$ must satisfy the following equations:
\begin{align}
    \mu\left(\ptwo\right) & = \lambda^2 \label{ptwo}\\
    \mu\left(\dtwo\right)+ \mu\left(\ptwo\right)& = 1 \label{dtwo}\\
    \mu\left(\pthree\right) &= \lambda^3 \label{pthree}\\
    \mu \left(\elle\right)+ \mu\left(\pthree\right) & = \lambda(1-\lambda^2) \label{elle}\\
    \mu\left( \ptwodot\right) &=\frac{ \mu\left(\elle\right)^2}{(1-\lambda^2)} \label{ptwodot}\\
     \mu\left(\mahT\right) & = \lambda^4 \label{T}\\
     \mu\left(\ptwodot\right)+\mu\left(\mahT\right) & =(1-\lambda^2)^2\lambda^2.\label{contreq}
\end{align}
\end{prop}
\begin{proof}
Equation \ref{contreq} is obtained with a different technique from the rest. The proof for equations \ref{ptwo}-\ref{T} consists in identifying various free amalgamations in the graphs involved and in the repeated use of Corollary \ref{noextcor} or the fact that all eventual closures we will consider are obtained by iterating one-point extensions and using Theorem \ref{Hrushcor}. Again, to obtain these equations, we consider the eventual closures $D_i$ of a free amalgamation $B\amalg_A C$. For clarity, we shall illustrate this through pictures colouring differently the components of $B\amalg_A C$ and the one-point extensions. We colour the copy of $A$ in purple,  the copy of $B\setminus A$ in blue and the copy of $C\setminus A$ in orange. When dealing with one-point extensions we shall colour the additional points in green. If you are reading the paper in black and white printing, the copy of $B\setminus A$ appears black, the copy of $A$ appears grey, the copy of $C\setminus A$ appears light grey and any one-point extension will be very light grey. We also label the vertices for extra clarity.\\

Let us begin with equation \ref{ptwo}. Consider a path of length two as a free amalgamation of two paths of length one over a point as in Figure \ref{fig:ptwo}. This free amalgamation has no further eventual closures. Since the measure of an edge is $\lambda$ and the measure of a point is 1, equation \ref{ptwo} follows from Corollary \ref{noextcor}.
\begin{figure}[H]
    \centering
\includegraphics[]{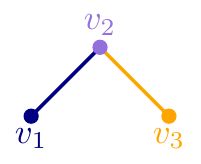}
    \caption{We consider a path of length two as the free amalgamation of the two edges $B:=v_1v_2$ and $C:=v_2 v_3$ over the intermediate vertex $A:=v_2$. Note that this path has already maximal size for its dimension  as $f(3)=4$ and so there are no further eventual closures.}
    \label{fig:ptwo}
\end{figure}

Now, for equation \ref{dtwo}, consider the free amalgamation of two points over the empty set. This does have a one point closure, namely a copy of a path of length two (by joining the vertex of the extension to the two distinct vertices in the amalgamation). There cannot be further extensions since $f(3)=4$, and so the resulting graph has maximal size for its dimension. Since vertices have measure one in $\mathcal{M}_f$,  equation \ref{dtwo} follows. Note that as a consequence of equations \ref{dtwo} and \ref{ptwo}, we obtain that 
\[\mu\left(\dtwo\right)=(1-\lambda^2),\]
and we shall use this fact in future equations.

Equations (\ref{pthree})-(\ref{T}) are explained in Figures \ref{fig:pthree}-\ref{fig:T}.

\begin{figure}[H]
    \centering
\includegraphics[]{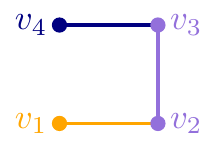}
    \caption{We consider a path of length three $v_1v_2v_3v_4$ as a free amalgamation of two paths of length two, $B:=v_1 v_2 v_3$ and $C:=v_2 v_3 v_4$, over the path of length one $A:=v_2 v_3$. There are no further eventual closures since $f(4)=5$ and so we have maximal size for this dimension. This amalgamation yields equation \ref{pthree} recalling that the two paths of length two in the amalgamation have measure $\lambda^2$, and that the path over which they are amalgamated has measure $\lambda$.}
    \label{fig:pthree}
\end{figure}

\begin{figure}[H]
    \centering

\begin{minipage}{0.4\textwidth}
\centering
\includegraphics[]{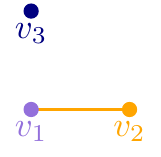}
\end{minipage}
\begin{minipage}{0.4\textwidth}
\centering
\includegraphics[]{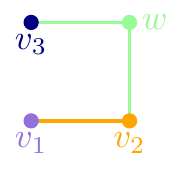}
\end{minipage}

    \caption{The disjoint union of an edge $v_1 v_2$ and a vertex $v_3$ corresponds to the free amalgamation of the disjoint union of two vertices $B:=v_1 v_3$ and of the edge $C:=v_1 v_2$ over the vertex $A:=v_1$. We can obtain a one point closure by joining a vertex $w$ to the points $v_2$ and $v_3$. It is the unique one point closure since $\vert B\setminus A\vert =\vert C\setminus A\vert =1$. There are no larger eventual closures since $f(4)=5$. Thus, we get equation \ref{elle}, recalling that the measure of two disjoint points is $(1-\lambda^2)$.}
    \label{fig:elle}
\end{figure}

\begin{figure}[H]
    \centering
\includegraphics[]{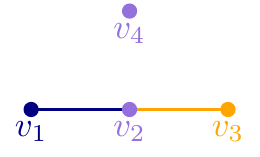}
\caption{We think of the disjoint union of a vertex $v_4$ and a path of length two $v_1 v_2 v_3$ as a free amalgamations of two copies of the disjoint union of an edge and a vertex, $B:=v_1 v_2 v_4$ and $C:=v_2v_3 v_4$, over the disjoint union of two vertices $A:=v_2v_4$. The free amalgamation has dimension $6$. Since $f(6)=6$ and the vertices in $B\setminus A$ and $C\setminus A$ are at distance two from each other, from Corollary \ref{noextcor} we obtain the desired equation \ref{ptwodot}.}
    \label{fig:ptwodot}
\end{figure}

\begin{figure}[H]
    \centering
\includegraphics[]{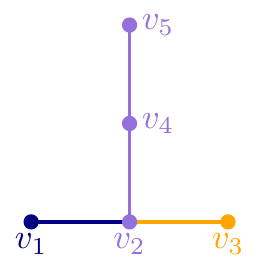}
    \caption{We consider the free amalgamation of two paths of length three, $B:=v_1 v_2 v_4 v_5$ and $C:=v_3 v_2 v_4 v_5$ over a path of length two $A:=v_2 v_4 v_5$. Equation \ref{T} follows similarly to the above by Corollary \ref{noextcor}.}
    \label{fig:T}
\end{figure}

We are now ready to obtain equation \ref{contreq}. For it we can use Corollary \ref{triang} since the structures we are working with have weak elimination of imaginaries, as proven in the Appendix \ref{appendix}. Since in $\mathcal{M}_f$ types of finite tuples are determined by the quantifier-free type of their closures, there are only two $2$-types over the emptyset, $p_1(x,y)$ and $p_2(x,y)$, such that $p_i(a,b)$ implies $a\indep{}{}b$. These are isolated by the formulas $\phi(x, y)$ and  $\psi(x, y)$, saying respectively that $x$ and $y$ are at distance two from each other and that they are at distance $>2$. We consider the measure of $\psi(x_1, x_2)\wedge\psi(x_2, x_3)\wedge \phi(x_1, x_3)$. By Corollary \ref{triang}, we have
\[\mu(\psi(x_1, x_2)\wedge\psi(x_2, x_3)\wedge \phi(x_1, x_3))=\mu(\psi(x_1, x_2))^2\mu(\phi(x_1, x_3))=(1-\lambda^2)^2\lambda^2.\]

We wish to express $\psi(x_1, x_2)\wedge\psi(x_2, x_3)\wedge \phi(x_1, x_3)$ as a disjoint union of $3$-types of maximal dimension in order to expand the left hand side of this equation. Let us look at the complete $3$-types $\mathrm{tp}(a_1a_2a_3)$ such that $a_1a_2a_3\vDash \psi(x_1, x_2)\wedge\psi(x_2, x_3)\wedge \phi(x_1, x_3)$ and $a_i\indep{}{} a_j a_k$ for $\{i,j,k\}=\{1,2,3\}$ distinct. Let $A$ be the graph in $\mathcal{K}_f$ given by the following picture:
\begin{figure}[H]
    \centering
\includegraphics[]{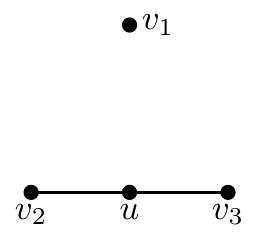}
    \caption{The graph $A$ consists of the disjoint union of a path of length two $v_2 u v_3$ and a vertex $v_1$. Note that if $A\leq M_f$, then $\mathcal{M}_f\vDash \psi(v_1, v_2) \wedge \psi(v_2, v_3) \wedge \phi(v_1, v_3)$ and $v_i$ is independent from $v_j v_k$ for $\{i,j,k\}=\{1,2,3\}$.}
    \label{fig:A}
\end{figure}
Again, since the types of these triples are determined by the quantifier-free types of their closures, we may focus on the graphs $B$ in $\mathcal{K}_f$ such that, $A\leq^* B$, $\delta(B)=\delta(A)=6$, and $v_1v_2, v_1v_2, v_2uv_3\leq B$. Note that since $\delta(B)=6$, $B$ may have at most $6$ points and must be obtained by iterating one-point extensions on the graph $A$.

\begin{figure}[H]
    \centering
 \begin{minipage}{0.45\textwidth}
\centering
\includegraphics[]{Pictures/A1.pdf}
 \end{minipage}
 \begin{minipage}{0.45\textwidth}
 \centering
\includegraphics[]{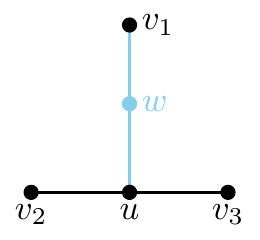}
 \end{minipage}
    \caption{Note that since $v_1v_2, v_1v_2, v_2uv_3\leq B$, any one-point extension of $A$ by a vertex $w$ must be such that $w$ has an edge with $u$. Since we avoid triangles, the only possible one-point extension is given by joining $w$ to $u$ and $v_1$. Call the graph so obtained $B'$. Note that $B'$ doesn't have any one-point extension since any two points of $B'$ have distance $\leq 3$ from each other and $\mathcal{M}_f$ avoids cycles of size $<6$.}
    \label{fig:contreq}
\end{figure}

Let $A(v_1, v_2, v_3, u)$ be the formula expressing $v_1 v_2 v_3 u \cong A\leq \mathcal{M}_f$ and $B'(v_1, v_2, v_3, u, w)$ be the formula expressing $v_1 v_2 v_3 u w\cong B'\leq \mathcal{M}_f$, where $B'$ is described in Figure \ref{fig:contreq}. From the argument in Figure \ref{fig:contreq}, we can see that there are only two $3$-types $p(x_1, x_2, x_3)$ which are completions of $\psi(x_1, x_2)\wedge\psi(x_2, x_3)\wedge \phi(x_1, x_3)$ and which satisfy our independence requirements, and they are isolated by $\exists u A(x_1, x_2, x_3, u)$ and $\exists u, w B'(x_1, x_2, x_3, u, w)$. Hence, by additivity, 
\[\mu(\exists u A(x_1, x_2, x_3, u))+\mu(\exists u, w B'(x_1, x_2, x_3, u, w))=\mu(\psi(x_1, x_2)\wedge\psi(x_2, x_3)\wedge \phi(x_1, x_3)).\]

Consider the projection $\pi:M_f^4\to M_f^3$ onto the first three coordinates. The restriction of this map to $A(M^4)$ is injective, and so by Fubini and algebraicity,
\[\mu(\exists u A(x_1, x_2, x_3, u))=\mu(A(x_1, x_2, x_3, u))=\mu\left(\ptwodot\right).\]
Similarly, considering the projection $\pi':M_f^5\to M_f^3$ onto the first three coordinates and $B(M^5)$,
\[\mu(\exists u, w B'(x_1, x_2, x_3, u, w))=\mu(B'(x_1, x_2, x_3, u, w))=\mu\left(\mahT\right).\]
This yields the desired equation. 
\end{proof}

\begin{theorem}\label{eqnstheorem} The structures of the form $\mathcal{M}_f$ built following Construction \ref{actualf} are not $MS$-measurable. Hence, there are $\omega$-categorical supersimple finite $SU$-rank structures which are not $MS$-measurable. Indeed, there are supersimple $\omega$-categorical structures of finite $SU$-rank and with independent $n$-amalgamation for all $n$ which are not $MS$-measurable.
\end{theorem}
\begin{proof}
By Corollary \ref{SUok}, we know that if $\mathcal{M}_f$ is $MS$-measurable, it also has an $MS$-dimension-measure where the dimension is given by $SU$-rank. However, the equations from Proposition \ref{manyeqns} imply that $\lambda=0$. This contradicts the positivity assumption for the measure in $MS$-measurable structures. To see this note that equations \ref{pthree} and \ref{elle} imply that
\[\mu\left(\elle\right)=\lambda-2\lambda^3.\]
From equations \ref{ptwodot}, \ref{T} and \ref{contreq} we then get that
\[\frac{\lambda^2(1-2\lambda^2)^2}{(1-\lambda^2)}+\lambda^4=(1-\lambda^2)^2\lambda^2.\]
Simplifying, we obtain that $\lambda=0$.
\end{proof}

Whilst we had a particular choice for $\alpha$ and for the initial values of $f$, it is plausible that similar results should hold for different choices of both. This raises the question of whether in general, non-trivial $\omega$-categorical Hrushovski constructions are not $MS$-measurable. This would provide further evidence for the conjecture that $MS$-measurable $\omega$-categorical structures are one-based.\\

In subsequent work \cite{Ergome}, we find more general reasons for which various classes of $\omega$-categorical Hrushovski constructions are not $MS$-measurable by studying invariant Keisler measures in these structures. In particular, we prove that $\omega$-categorical $MS$-measurable structures must satisfy a stronger version of the independence theorem and that the equation in Corollary \ref{triang} holds also when one of the pairs is weakly algebraically independent. We also prove that in the structures $\mathcal{M}_f$ we introduced, the formula asserting that "$x$ has distance two from $a$" does not fork over the empty-set but is universally measure zero. While our results make it implausible, it remains an open question whether any not one-based $\omega$-categorical Hrushovski construction is $MS$-measurable. 

\section{Appendix: Proving the main properties of our Hrushovski constructions}\label{appendix}

In Section \ref{finally}, we introduced a class of $\omega$-categorical supersimple finite $SU$-rank Hrushovski constructions and proved they are not $MS$-measurable. However, we omitted from the body of this article the proofs that our choice of $f$ yields such structures. Indeed, the proof of simplicity requires some effort. We include these technical results in this appendix.\\

We begin with an abstract discussion of how to prove simplicity in Subsection \ref{simpH}. Then, in Subsection \ref{example} we move to proving that our choice of $f$ in Construction \ref{actualf} yields $\omega$-categorical Hrushovski constructions and that these are supersimple of finite $SU$-rank, with weak elimination of imaginaries. We further note that these can be built so that they may satisfy independent $n$-amalgamation for all $n$.

\subsection{Simplicity of \texorpdfstring{$\omega$}{omega}-categorical Hrushovski constructions}\label{simpH}
In this subsection, we discuss under which conditions on $f$ $\omega$-categorical Hrushovski constructions are supersimple of finite $SU$-rank. This was first explored in \cite{Udiamalg}. We simplify some of the conditions for supersimplicity discussed in \cite{Wong} and \cite{supersimple}. This will shorten the proofs of supersimplicity of the structures considered in the next section. Much of this material is implicit in the proof of Theorem 3.6 in \cite{supersimple}. Here we make the arguments explicit and correct a mistake in Remark 3.8 of the same article.\\

Let us begin by reminding some basic properties of the d-closed and self-sufficient relations \cite[Lemma 3.10]{Evans}:

\begin{lemma}\label{basicleq} Let $C\in\overline{\mathcal{K}}$ and let $\leq'$ stand for either $\leq$ or $\leq^*$. Then, the following hold:
\begin{enumerate}
    \item Let $A\leq' C$ be finite, and $B\subseteq C$. Then, $A\cap B\leq'B$.
    \item Let $A$ and $B$ be finite such that $A\leq'B\leq'C$. Then, $A\leq'C$.
    \item Let $A$ and $B$ be finite with $A, B\leq'C$. Then, $A\cap B\leq'C$.
\end{enumerate}
\end{lemma}

\begin{definition}[Independence Theorem Diagram]\label{ITDdef} Let $D$ be a graph.  Suppose that $D$ has subgraphs $D_i$ and $D_{ij}$ for $0\leq i<j\leq 3$ all contained in $\mathcal{K}_f$. To simplify our notation, we allow inverting indices, e.g. $D_{12}=D_{21}$. 
We say that $(D; D_i; D_{ij})$ is an \textbf{independence theorem diagram} (ITD) with respect to $\mathcal{K}_f$ if the following hold:
\begin{itemize}
    \item $D_{0j}=D_j$ for $1\leq j\leq 3$;
    \item $D_i\cap D_j=D_0$ for $1\leq i<j\leq 3$;
    \item $D_i, D_j\leq D_{ij}$ for $1\leq i<j\leq 3$; 
    \item $D_{ij}\cap D_{jk}=D_j$ for $\{i,j,k\}=\{1,2,3\}$;
    \item $D_i$ and $D_j$ are independent over $D_0$ in $D_{ij}$ for $1\leq i<j\leq 3$;
    \item Any edge in $D$ is entirely contained within some $D_{ij}$ for $1\leq i<j\leq 3$.
\end{itemize}
Figure \ref{fig:2} gives a visual representation of an independence theorem diagram.

We say that $(D; D_i; D_{ij})$ is a \textbf{proper ITD} when it satisfies the following additional conditions:
\begin{itemize}
\item $D_0 \subsetneq D_i$ for $1\leq i\leq 3$;
\item $D_i\cup D_j\subsetneq D_{ij}=\mathrm{cl}_D(D_i\cup D_j)$ for $1\leq i<j\leq 3$.
\end{itemize}

We say that $\mathcal{K}_f$ is \textbf{closed under ITDs}  if, whenever $(D; D_i; D_{ij})$ is an ITD, we have that $D\in\mathcal{K}_f$. 
\end{definition}
\begin{figure}
    \centering
    \includegraphics[width=13cm]{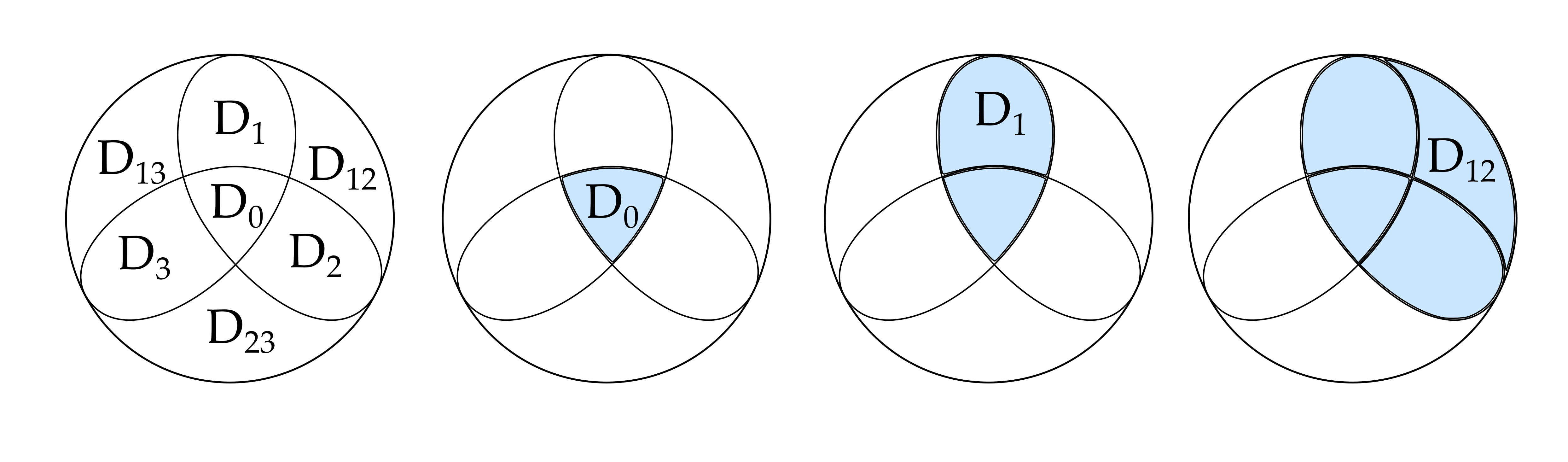}
    \caption{A visual representation of an independence theorem diagram. The first picture represents an ITD with its various parts labelled. To distinguish the parts more clearly, the last three pictures show as a shaded area $D_0$, $D_1$, and $D_{12}$ respectively.}
    \label{fig:2}
\end{figure}

We note some basic properties of independence theorem diagrams and proper ITDs.

\begin{lemma} Suppose $(D;D_i;D_{ij})$ is an ITD for $\mathcal{K}_f$. Then, 
\begin{enumerate}
    \item $D_0\leq D_{ij}$ for $0\leq i<j\leq 3$;
    \item $D_{ij}\leq D$;
    \item $D_{ij}\leq D_{ij}\cup D_{jk}$;
    \item $D_{ij}\cup D_{jk}\leq^* D$.
\end{enumerate}
\end{lemma}
\begin{proof} Part 1 is a direct consequence of Lemma \ref{basicleq}.1. Parts 2,3, and 4 are from \cite[Lemma1.6]{Wong}.
\end{proof}

From \cite[Corollary 2.24 \& Theorem 3.6]{supersimple} we know that: 

\begin{theorem}\label{supersimpth} Suppose that $\alpha\in\mathbb{N}$ and $\mathcal{K}_f$ is closed under ITDs. Then, $\mathcal{M}_f$ is supersimple of finite $SU$-rank. 
\end{theorem}

Hence, it is sufficient to include all ITDs in $\mathcal{K}_f$ in order to have supersimplicity and finite $SU$-rank for $\mathcal{M}_f$. However, this is not easy to check because in order to see whether $D$ is in $\mathcal{K}_f$ we need to verify that every subgraph $B\subset D$ is also in $\mathcal{K}_f$. The following lemmas help us simplifying this process. 

\begin{lemma} Let $C\leq D\in\mathcal{K}_f$, where $D$ is an ITD. Then, $C$ is an ITD. 
\end{lemma}
\begin{proof}
Call $C_{ij}$ the intersection of $C$ and $D_{ij}$. The only conditions that we need to check are that $C_i\leq C_{ij}$ and that $C_i$ and $C_j$ are independent over $C_0$ in $C_{ij}$. 

Now, the first condition follows by Lemma \ref{basicleq}.1 since $D_i\leq D_{ij}$ and $C_{ij}\subseteq D_{ij}$ implies $C_i=D_i\cap C_{ij}\leq C_{ij}$.

For the second condition, first note that since $D_i$ and $D_j$ are freely amalgamated over $D_0$ in $D_{ij}$, so must be $C_i$ and $C_j$ over $C_0$ in $C_{ij}$, being just intersections of these sets. We know that $D_i\cup D_j\leq^*\mathrm{cl}_{D_{ij}}(D_i\cup D_j), \mathrm{cl}_{C_{ij}}(C_i\cup C_j)\cap (D_j\cup D_j)= C_i\cup C_j$, and $\mathrm{cl}_{C_{ij}}(C_i\cup C_j)\cap\mathrm{cl}_{D_{ij}}(D_i\cup D_j)=\mathrm{cl}_{C_{ij}}(C_i\cup C_j)$. Hence, by Lemma \ref{basicleq}.1, $C_i\cup C_j\leq^*\mathrm{cl}_{C_{ij}}(C_i\cup C_j)$. But then, by Lemma \ref{indepfree}, we know that $C_i$ and $C_j$ are independent over $C_0$ in $C_{ij}$.
\end{proof}

Hence, in proving supersimplicity, we can reduce the number of graphs for which we check $f(\vert  D\vert )\leq \delta(D)$ by only focusing on proper ITDs.

\begin{lemma}\label{strongITD} The following are equivalent:
\begin{enumerate}
    \item $\mathcal{K}_f$ is closed under ITDs;
    \item $\mathcal{K}_f$ is closed under proper ITDs;
    \item For each proper ITD $D$ we have that $f(\vert D\vert )\leq\delta(D)$.
\end{enumerate}

\end{lemma}

\begin{proof}
(1)$\Rightarrow$(2)$\Rightarrow$(3) follows from the definitions of strong ITD and $\mathcal{K}_f$. Let us focus then on (2)$\Rightarrow$(1). Let $D$ be an ITD. For it to be in $\mathcal{K}_f$ we need to show that for any $C\subseteq D, f(\vert C\vert )\leq\delta(C)$. Note that if $f(\vert C\vert )>\delta(C)$, also $\mathrm{cl}_D(C)$ satisfies this condition since
\[f(\vert \mathrm{cl}_D(C)\vert )\geq f(\vert C\vert )>\delta(\vert C\vert )\geq \delta(\mathrm{cl}_D(C)),\]
where the first inequality holds since $f$ is increasing and the last by definition of closure. Hence, we may assume that $C\leq D$.
We shall show that if $C$ is not already a proper ITD, it must already be in $\mathcal{K}_f$, or that there is a proper ITD $C'$, such that if $C'$ is in $\mathcal{K}_f$, then so is $C$. 

We label the intersections of $C$ with $D$ so that $C_i:=D_i\cap C$ and $C_{ij}:=D_{ij}\cap C$. By the previous lemma $C$ is an independence theorem diagram. 

Suppose that $C_1=C_0$, then $C$ is in $\mathcal{K}_f$, being obtained by freely amalgamating $C_{12}$ and $C_{23}$ over $C_2$, and then freeely amalgamating this structure with $C_{13}$ over $C_3$. Hence, if $C_i=C_0$ for any $1\leq i\leq 3$, $C\in\mathcal{K}_f$. 

Suppose that $C_{12}=C_1\cup C_2$. Then, $C$ is the free amalgamation of $C_{23}$ and $C_{13}$ over $C_3$. Thus, if $C_{ij}=C_i\cup C_j$, then $C\in\mathcal{K}_f$.

Finally, consider the subgraph of $C'$ of $C$ which constitutes the proper independence theorem diagram obtained by taking $C'_{ij}:=\mathrm{cl}^C(C_i\cup C_j)$. Now, by construction and definition of closure, we must have $C'_{ij}\leq C_{ij}$. Furthermore, $C'_{ij}\leq C'$ 
Hence, $C$ may be obtained as a free amalgamation of $C'$ and $C_{ij}$ over $C'_{ij}$ (eventually repeating this operation for the different $C_{ij}$'s). 

So, we have seen that if all proper ITDs are in $\mathcal{K}_f$, then so are all ITDs, and so (2)$\Rightarrow$(1). We prove that (3)$\Rightarrow$(2) holds by considering a minimal counterexample. Suppose that $D$ is a minimal proper ITD not in $\mathcal{K}_f$. Since $f(\vert D\vert )\leq \delta(D)$, there must be some $C\subset D$ such that $f(\vert C\vert )>\delta(\vert C\vert )$.
 Again, we may assume that $C\leq D$. By the previous lemma, $C$ is an ITD. From the steps above, we know that either it will be either a free amalgamation of graphs in $\mathcal{K}_f$, or the free amalgamation of graphs in $\mathcal{K}_f$ and some proper ITD, or it will be a proper ITD again. But, by minimality, any of these cases implies that $C\in\mathcal{K}_f$. So, it is sufficient to check the condition $f(\vert D\vert )\leq \delta(D)$ in proper ITDs. 
\end{proof}

\begin{lemma}\label{SITDcond} Let $\alpha\in\mathbb{N}$, $\mathcal{K}_f$ be a free amalgamation class. Let $n_1\in\mathbb{N}$. Suppose that for $n_1\leq t$, $f(3t)\leq f(t)+k$, for fixed $k\in \mathbb{N}$. Let $D$ be an ITD. Without loss of generality, say that $D_{12}$ is the $D_{ij}$ of biggest predimension, and call its predimension $d_{12}$. Suppose that $f(n_1) \leq d_{12}$. We have that if $\delta(D)\geq d_{12}+k$, then $f(\vert D\vert )\leq \delta(D)$. 
\end{lemma}
\begin{proof}
This proof is substantially identical to the final part of the Theorem 3.6 in \cite{supersimple}. However, we repeat the argument for completeness and in order to avoid confusion with Remark 3.8 \cite{supersimple}, which follows the theorem and contains a mistake.

Let $g$ be the inverse of $f$. Making the substitution $g(x)=t$ into $f(3t)\leq f(t)+k$, we get that $3g(x)\leq g(x+k)$ for $f(n_1)\leq x$. Then, for $f(n_1)\leq d_{12}$:

\[ \vert D\vert \leq \sum_{1\leq i<j\leq3} \vert D_{ij}\vert  \leq \sum_{1\leq i<j\leq3} g(\delta(D_{ij})) \leq 3 g(d_{12})\leq g(d_{12}+k)\leq g(\delta(D)) \]
Where the second inequality holds since $D_{ij}\in\mathcal{K}_f$ and the fourth holds since $\delta(D)\geq d_{12}+k$ and $g$ is increasing. Note that the resulting inequality is equivalent to $f(\vert D\vert )\leq \delta(D)$. 
\end{proof}

Hence, we have a method to obtain supersimple $\omega$-categorical Hrushovski constructions of finite $SU$-rank. And now we know how to verify this more easily.

\subsection{The structures \texorpdfstring{$\mathcal{M}_f$}{Mf} and their supersimplicity}\label{example}

In this subsection, we focus on our choice of structures $\mathcal{M}_f$  built as specified in Construction \ref{actualf}. We prove that these structures have the various properties mentioned in the main article. \\

We begin by noting that $\mathcal{K}_f$ is a free amalgamation class and so the structures $\mathcal{M}_f$ may be built according to Theorem \ref{Hrushconstrtheorem}. 

\begin{prop}\label{freeam} The class $\mathcal{K}_f$ is a free amalgamation class. 
\end{prop}
\begin{proof}

Since $f$ is a good function for $t\geq 6$, by Lemma \ref{showfreeam}, we just need to show that for $A_0\leq A_1, A_2\in\mathcal{K}_f$, then $A_1\amalg_{A_0}A_2\in\mathcal{K}_f$ for $\vert A_i\vert \leq 6$. Note that $\vert A_1\amalg_{A_0}A_2\vert =\vert A_1\vert +\vert A_2\vert -\vert A_0\vert $, and $\delta(A_1\amalg_{A_0}A_2)=\delta(A_1)+\delta(A_2)-\delta(A_0)$. So, in the context of Figure \ref{fig:newf}, it is sufficient to check that given any three dots $p,q$, and $r$ (possibly with $q=r$) lying above $f$, with $p=(p_1,p_2), q=(q_1,q_2), r=(r_1,r_2)$ such that $p_1<q_1, r_1$ and $p_2<q_2, r_2$, the fourth vertex of the parallelogram with edges $\overline{pq}$ and $\overline{pr}$ is still above the function $f$. Since $f(18)\leq 7$ and our function is increasing, this can be easily verified.
\end{proof}

Now we prove that $\mathcal{M}_f$ is supersimple of $SU$-rank 2. We proceed step by step by proving that for any proper ITD $D$ for $\mathcal{K}_f$, $f(\vert D\vert )\leq \delta(D)$. Hence, the proof will follow by Lemma \ref{strongITD} and Theorem \ref{supersimpth}. We shall adopt the notation we already set for proper ITDs. To avoid confusion when speaking of the various element of an ITD, we always write $D$ for the proper ITD, we write $D_{ij}$ only for $i,j\in\{1,2,3\}$, and consistently write  $D_j$ for $D_{0j}$. Furthermore, given $D$, we shall assume without loss of generality that $D_{12}$ has maximal dimension among the $D_{ij}$. We shall also write $d_i$ for $\delta(D_i)$, $d_{ij}$ for $\delta(D_{ij})$ and $d$ for $\delta(D)$ to simplify our notation.

\begin{lemma} For all proper ITDs $D$ for $\mathcal{K}_f$ with $d_{12}\leq 4$, $f(\vert D\vert )\leq \delta(D)$. For this proof we only need to assume that $f(1)=2, f(2)=3, f(3)=4$ and $f(6)\leq 6$.
\end{lemma}
\begin{proof}
For $d_{12}\leq 4$, $\vert D_{12}\vert \leq 3$, and so $\vert D_{ij}\vert \leq 3$. Since we require $D_i, D_j\neq \emptyset$ and $D_i\cup D_j\subsetneq D_{ij}$ (for $i\neq j\in\{1,2,3\}$), we must have that $\vert D_{ij}\vert =3$ for each $i\neq j\in\{1,2,3\}$, $\vert D_i\vert =1$ for each $i\in\{1,2,3\}$, and $\vert D_0\vert =0$. There is only one graph satisfying these requirements, i.e. a $6$-cycle as shown in Figure \ref{fig:hexagon}.
\begin{figure}[H]
    \centering
  \includegraphics[]{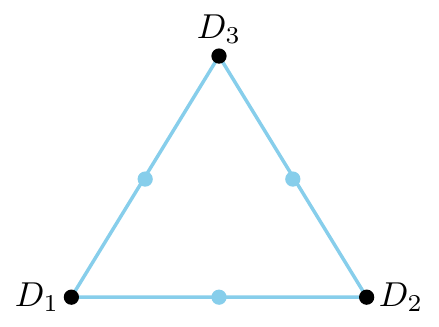}
 \caption{Each $D_i$ for $1\leq i\leq 3$ is a vertex. We can see that each $D_{ij}$ must be a copy of a path of length 2 since $\vert D_{ij}\vert =3$ and $d_{ij}\leq 4$. The resulting proper ITD is a $6$-cycle.}
    \label{fig:hexagon}
\end{figure}
\end{proof}

\begin{lemma} For all proper ITDs $D$ for $\mathcal{K}_f$ with $d_{12}\leq 5$, $f(\vert D\vert )\leq\delta(D)$. We only need to assume $f(1)=2, f(2)=3, f(3)=4, f(4)=5$, and $f(6)\leq 6$, and $f(12)\leq 7$.
\end{lemma}
\begin{proof} Given the previous lemma, we only need to prove the condition for $d_{12}=5$.
Note that since $f(3\cdot 4)=f(12)\leq 7=f(4)+2$, by Lemma \ref{SITDcond} we only need to check the case of $d=6$.

 We have that $d_0>0$ since otherwise, the condition that $d-d_{12}=1$ forces $d_3=1$, which is impossible since no graph in $\mathcal{K}_f$ has predimension $1$. Hence, we need to check the cases of $d_0=2$ and $d_0=3$ (note that for $d_0>3$ we cannot have $d=6$ since $d>d_0+3$). By definition of an ITD, the following inclusion-exclusion formula holds:
\begin{equation}
    \vert D\vert =\sum_{1\leq i<j\leq3} \vert D_{ij}\vert -\sum_{1\leq i\leq3} \vert D_i\vert +\vert D_0\vert .
\end{equation}
Knowing the upper bounds for the sizes of the $D_{ij}$ and the lower bounds for the sizes of the $D_i$, we can come to an upper bound for $D$:
\begin{equation}
    \vert D\vert \leq \sum_{1\leq i<j\leq3} \lfloor f^{-1}(d_i+d_j-d_0) \rfloor -\sum_{1\leq i\leq3} \min\{\vert B\vert  \text{ s.t. } \delta(B)=d_i\}+\vert D_0\vert :=\beta
\end{equation}

Furthermore, note that $(d_1-d_0)+(d_2-d_0)+d_0=d_{12}$ and that $d_3-d_0=d-d_{12}=1$. 

For $d_0=2$, without loss of generality we have $d_1=4, d_2=3, d_3=3$ (since we must have $d_1-d_0=2, d_2-d_0=1, d_3-d_0=1$). Hence, we get that $\beta=4\cdot 2+3-2\cdot 3+1=6$, and so $\vert D\vert \leq 6$. Since $f(6)\leq 6$, $f(\vert D\vert )\leq 6=d$.

For $d_0=3$, we obtain that $d_1-d_0=1, d_2-d_0=1, d_3-d_0=1$. Since $d_0=3$ implies that $\vert D_0\vert =2$, $\vert D_i\vert \geq 3$ so $\beta=3\cdot 4-3\cdot 3+2=5$, and so $f(\vert D\vert )\leq f(5)<6=d$, as desired.
\end{proof}

\begin{theorem}\label{sups} Let $\mathcal{K}_f$ and $\mathcal{M}_f$ be as in Construction \ref{actualf}. Then, $\mathcal{M}_f$ is supersimple of finite $SU$-rank. In particular, it has $SU$-rank $2$.
\end{theorem}
\begin{proof} From Theorem \ref{supersimpth} and Lemma \ref{strongITD} we need to check that for any proper ITD $D$, $f(\vert D\vert )\leq \delta(D)$. We know from the lemmas above for any proper ITD $D$ for $\mathcal{K}_f$ with $d_{12}\leq 5$, $f(\vert D\vert )\leq \delta(D)$. Since for $t\geq 6, f(3\cdot t)\leq f(t)+1$, by Lemma \ref{SITDcond} we have that for $d_{12}\geq 6$, if $\delta(D)\geq d_{12}+1$, then $f(\vert D\vert )\leq \delta(D)$. But $\delta(D)> d_{12}$ in any proper ITD. Hence, any proper ITD with $d_{12}\geq 6$ is such that $f(\vert D\vert )\leq \delta(D)$ and so the theorem follows. Finally, $SU$-rank coincides with the Hrushovski dimension, and so $SU(\mathcal{M}_f)=2$. This follows from the characterisation of non-forking independence in terms of the dimension \cite[Corollary 2.21]{supersimple}.
\end{proof}

In the process of our proof of Theorem \ref{sups} we have proven that the smallest proper ITD in $\mathcal{K}_f$ has predimension 6. Our conditions in Construction \ref{actualf} put no constraints on how slowly $f(t)$ grows for $t> 6$. Hence, we can see that we may make $f$ slow growing enough to satisfy independent $n$-amalgamation for any $n\in\mathbb{N}$ of our choice. Indeed, by chosing $f(t)$ with growth the inverse of $(n+1)!$ for $t\geq 6$, we can choose $\mathcal{M}_f$ to have independent $n$ amalgamation for all $n$ \cite{Udiamalg}.\\

Finally, we note that structures of the form of $\mathcal{M}_f$ have weak elimination of imaginaries.

\begin{prop}\label{HWEI}
Let $\mathcal{M}_f$ be as in Remark $5.2$, then it has weak elimination of imaginaries. 
\end{prop}
\begin{proof}
From Theorem 5.12, we know that $\mathcal{M}_f$ satisfies conditions (P1)-(P5) from \cite{supersimple}. And these conditions are sufficient for weak elimination of imaginaries by Lemma 2.9 and Corollary 2.7 from \cite{supersimple}.   
\end{proof}

\printbibliography

\end{document}